\documentclass[preprint,1p]{elsarticle}
\usepackage{a4wide}
\usepackage{amsmath}
\usepackage{amssymb}
\usepackage{graphicx}
\usepackage{color}
\usepackage{epic}
\usepackage{tikz}                                                                                       
\usetikzlibrary{patterns}  
\usetikzlibrary{arrows}
\usepackage[labelfont=bf,skip=2pt]{caption}
\usepackage[bookmarks=false]{hyperref}
\usepackage[hang]{footmisc}

\usepackage{lipsum}
\makeatletter
\def\ps@pprintTitle{%
 \let\@oddhead\@empty
 \let\@evenhead\@empty
 \def\@oddfoot{}%
 \let\@evenfoot\@oddfoot}
\makeatother

\makeatletter
\def\printFirstPageNotes{%
  \iflongmktitle
   \let\columnwidth=\textwidth\fi
  \ifx\@tnotes\@empty\else\@tnotes\fi
  \ifx\@nonumnotes\@empty\else\@nonumnotes\fi
  \ifx\@cornotes\@empty\else\@cornotes\fi
  \ifx\@elseads\@empty\relax\else
   \let\thefootnote\relax
   \footnotetext{\ifnum\theead=1\relax
      \textit{E-Mail:\space}\else
      \textit{E-Mail:\space}\fi
     \@elseads}\fi
  \ifx\@elsuads\@empty\relax\else
   \let\thefootnote\relax
   \footnotetext{\textit{URL:\space}%
     \@elsuads}\fi
  \ifx\@fnotes\@empty\else\@fnotes\fi
  \iflongmktitle\if@twocolumn
   \let\columnwidth=\Columnwidth\fi\fi
}
\makeatother

\newdefinition{definition}{Definition}
\newtheorem{theorem}[definition]{Theorem}
\newtheorem{lemma}[definition]{Lemma}
\newdefinition{remark}[definition]{Remark}
\newdefinition{exmp}[definition]{Example}

\newdefinition{algorithm}[definition]{Algorithm}
\newtheorem{problem}[definition]{Approximation problem}
\newtheorem{assertion}[definition]{Assertion}
\newproof{proof}{Proof}
\newproof{prooftheorem}{Proof of Theorem \thetheorem}

\renewenvironment{proof}{{\bfseries Proof.}}{}
\renewenvironment{prooftheorem}{{\bfseries Proof}}{}

\begin{document}

\begin{center}
{\Large{}Discrete approximation by first-degree splines with free knots}
\end{center}

\vspace{0.2cm}
\begin{center}\setlength{\footnotemargin}{0.5em}
{\large Ludwig J. Cromme\footnote{Numerische und Angewandte Mathematik, BTU Cottbus-Senftenberg, Platz der Deutschen Einheit 1, D-03046 Cottbus, Germany, E-mail: Ludwig.Cromme@b-tu.de\rule[-8pt]{0pt}{5pt}} and Jens Kunath\footnote{Numerische und Angewandte Mathematik, BTU Cottbus-Senftenberg, Platz der Deutschen Einheit 1, D-03046 Cottbus, Germany}}
\end{center}

%\begin{frontmatter}
%\title{Discrete approximation by first-degree splines with free knots}
%\author[cromme]{Ludwig~J.~Cromme\corref{cor1}}
%\ead{Ludwig.Cromme@b-tu.de}
%\author[cromme]{Jens~Kunath}

%\cortext[cor1]{Corresponding author}
%\address[cromme]{Numerische und Angewandte Mathematik, BTU Cottbus-Senftenberg, Platz der Deutschen Einheit 1, D-03046 Cottbus, Germany}

\vspace{0.2cm}
\centerline{\hrulefill}
\vspace{-0.1cm}
\begin{flushleft}
\textbf{Abstract}
\end{flushleft}

\vspace{-0.2cm}\noindent
This paper deals with the approximation of discrete real-valued functions by first-degree splines (broken lines) with free knots for arbitrary $L_p$-norms ($1 \leq p \leq \infty)$. We prove the existence of best approximations und derive statements on the position of the (free) knots of a best approximation. Building on this, elsewhere we develop an algorithm to determine a (global) best approximation in the $L_2$-norm.

\vspace{0.25cm}
\noindent\textit{Key words}: splines of degree one, broken lines, splines with free knots, best approximation, discrete approximation, $L_p$-approximation ($1 \leq p \leq \infty$), existence theorem, first-degree splines

\centerline{\hrulefill}

%%%%%%%%%%%%%%%%%%%%%%%%%%%%%%%%%%%
%%%%%%%%%%%%%%%%%%%%%%%%%%%%%%%%%%%
% 1. Introduction
%%%%%%%%%%%%%%%%%%%%%%%%%%%%%%%%%%%
%%%%%%%%%%%%%%%%%%%%%%%%%%%%%%%%%%%

\section{Introduction and problem formulation}
\label{sec:intro}
First-degree splines (broken lines) with fixed knots have a series of notable properties: In interpolation they preserve convexity, positivity and monotonicity and are cheap to calculate (see for example \cite{kocic1997}). First-degree splines can also help fight the dreaded Gibbs phenomenon (see for example \cite{richards1991}). Broken lines are easily smoothed should smooth curves be needed (see for example \cite{koutsoyiannis2000}).

If approximation with a minimal number of knots of knots is wanted or the knots have a meaning from the user perspective, interpolation does not suffice and the knots must be positioned optimally. Thus, we have an approximation problem to solve.

In approximating a real-valued function of a real variable by first-degree splines with free knots, the existence of a best approximation in the continuous case is guaranteed (see for example Rice \cite{rice} for splines of degree $m \in \mathbb{N}$). Globally convergent numerical methods are not known.

In this paper the existence of a best approximation by first-degree splines with free knots for the \textit{discrete} case is shown and proof given for characteristic properties of a best approximation. A globally convergent numerical method based on this is derived in \cite{cromme01}.

Difficulties stem among others from the fact that a minimizing sequence bounded on a discrete domain can be unlimited if we consider the related sequence of continuous functions.

To begin with, we summarize the most important notations: Let $\mathbb{P}_m$ denote the real polynomials of degree smaller or equal $m$. Let $[a,b]$ be a real interval with $a < b$. For $a =: t_0 < t_1 < \ldots < t_k < t_{k+1} := b$ and $m, k \in \mathbb{N}$ 
\begin{equation*}
S^m(t_1,\ldots,t_k) := \left\{ \left. s \in C^{m-1}[a,b] \;\;\right|\;\; s|_{\left(t_j,t_{j+1}\right)} \in \mathbb{P}_m \;,\; j = 0,1,\ldots,k \right\}
\end{equation*} 
denotes the set of splines of degree $m$ with $k$ fixed (simple) knots $t_1,\ldots,t_k$. By the splines of degree $m \in \mathbb{N}$ with at most $k \in \mathbb{N}$ free (simple) knots we mean the set
\begin{eqnarray*}
S^m_k[a,b] &:=& \text{\Large $\left\{\right.$} s\in C^{m-1}[a,b] \;\; | \;\; \text{there exist points} \\
&& a =: t_0 < t_1 < \ldots < t_k < t_{k+1} := b\;\\ 
&& \text{with:} \;\; s|_{\left(t_i,t_{i+1}\right)} \in \mathbb{P}_m \;\;\text{for} \;\; i=0,1,\ldots,k  \text{\Large $\left.\right\}$} \quad.
\end{eqnarray*} 
Here, we call $t_j$ a \textit{proper} (or \textit{active}) knot of $s\in S^m(t_1,\ldots,t_k)$ or $s\in S^m_k[a,b]$ if the $m$-th derivate of $s$ has a jump discontinuity in $t_j$, otherwise $t_j$ is called \textit{improper} (or \textit{inactive}) knot.

\vspace{0.05in}
$S^m_k[a,b]$ is not closed. In $S^m_k[a,b]$ generally hence there exists no best approximation for a given $f\in C[a,b]$. Such is only assured in the closure of $S^m_k[a,b]$, i.e., if knots are allowed to coalesce, see Rice (\cite{rice}, Theorem 10-2). Hence, it is a feature of discrete approximation with broken lines that there are best approximations already in $S^1_k[a,b]$ as we will demonstrate below.

\vspace{0.05in}
Let $a =: x_0 < x_1 < \ldots < x_{\mu+1} := b$ be $\mu+2$ real abscissae and $f_0,\ldots,f_{\mu+1} \in \mathbb{R}$ values of a real function $f:[a,b] \rightarrow \mathbb{R}$. We define vectors $X := \left(x_0,\ldots,x_{\mu+1}\right)^t$, $F := \left(f_0,\ldots,f_{\mu+1}\right)^t$ and for arbitrary $g: \left\{x_0,\ldots,x_{\mu+1}\right\} \rightarrow \mathbb{R}$ set $g(X) := \left(g(x_0),\ldots,g(x_{\mu+1})\right)^t$. With the vector norm $\left\|\cdot\right\|_p$, $1 \leq p \leq \infty$, our approximation problem then reads:
\begin{problem}\label{defApproxproblem}
Determine $s^{\ast} \in S^1_k[a,b]$ with 
\begin{equation*}
\left\|f - s^{\ast}\right\|_{p,X} \;:=\; \left\|F-s^{\ast}(X)\right\|_p \;=\; \inf_{s \in S^1_k[a,b]} \left\|F-s(X)\right\|_p \quad.
\end{equation*}
\end{problem}

In what follows, we assume generally that $k \geq 1$ and $\mu \geq k+1 \geq 2$ since otherwise we are looking for a purely polynomial approximation or else the data can be reproduced exactly in $S^1_k[a,b]$.

\vspace{0.1in}
\begin{definition}\label{defMinimalfolge}
By a minimizing sequence of Problem \ref{defApproxproblem} we mean a sequence $\left(s^{(i)}\right)_{i\in\mathbb{N}}$ with $s^{(i)} \in S^1_k[a,b]$ and
\begin{equation*}
\left\|f-s^{(i)}\right\|_{p,X} \;\longrightarrow\; \inf_{s \in S^1_k[a,b]} \left\|f-s\right\|_{p,X}
\end{equation*}
for $i \rightarrow \infty$.
\end{definition}

% =================================
% Figure 1
% =================================

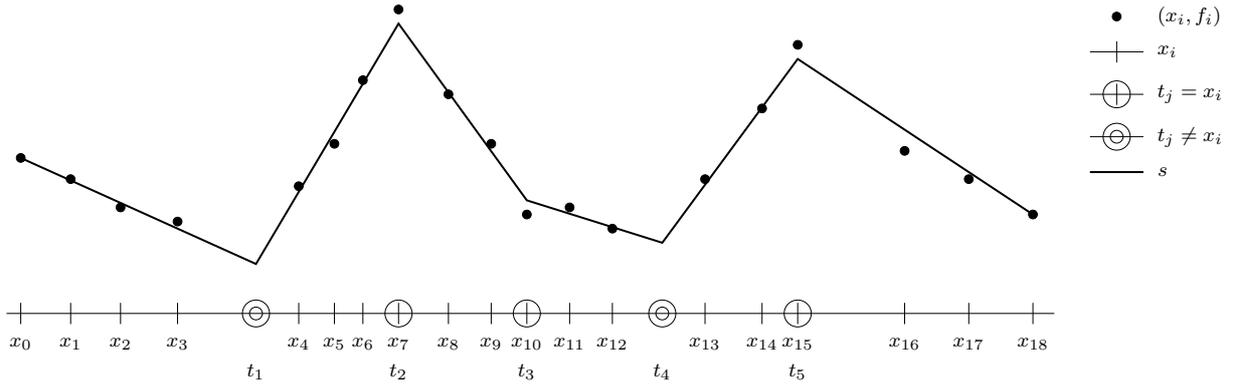
\begin{figure}[h]
\begin{center}
\resizebox{1.0\textwidth}{!}{
%\fbox{
\begin{tikzpicture}
\draw (-0.2,0)--(14.5,0);
\foreach \x/\y in {0/0,0.7/1,1.4/2,2.2/3,3.9/4,4.4/5,4.8/6,5.3/7,6/8,6.6/9,7.1/10,7.7/11,8.3/12,9.6/13,10.4/14,10.9/15,12.4/16,13.3/17,14.2/18} \draw (\x,0.15)--(\x,-0.15) node[below=2pt] {\footnotesize $x_{\y}$};
\draw (3.3,0) circle (0.09cm);
\draw (3.3,0) circle (0.19cm) node[below=0.6cm] {\footnotesize $t_1$};
\draw (5.3,0) circle (0.19cm) node[below=0.6cm] {\footnotesize $t_2$};
\draw (7.1,0) circle (0.19cm) node[below=0.6cm] {\footnotesize $t_3$};
\draw (9,0) circle (0.09cm);
\draw (9,0) circle (0.19cm) node[below=0.6cm] {\footnotesize $t_4$};
\draw (10.9,0) circle (0.19cm) node[below=0.6cm] {\footnotesize $t_5$};

% Polygonzug
\draw[thick,smooth] (0,2.2)--(3.3,0.7)--(5.3,4.1)--(7.1,1.6)--(9,1)--(10.9,3.6)--(14.2,1.4);

% Datenpunkte
\foreach \x/\y in {0/2.2,0.7/1.9,1.4/1.5,2.2/1.3,3.9/1.8,4.4/2.4,4.8/3.3,5.3/4.3,6/3.1,6.6/2.4,7.1/1.4,7.7/1.5,8.3/1.2,9.6/1.9,10.4/2.9,10.9/3.8,12.4/2.3,13.3/1.9,14.2/1.4} \draw[fill] (\x,\y) circle (1.75pt);

% Legende
%\draw (14.75,-1)--(14.75,4.5);
\draw[fill](15.375,4.2) circle (1.75pt);
\draw (15.75,4.2) node[right=2pt] {\footnotesize $(x_i,f_i)$};
\draw (15,3.7)--(15.75,3.7) node[right=2pt] {\footnotesize $x_i$};
\draw (15.375,3.85)--(15.375,3.55);
\draw (15,3.1)--(15.75,3.1) node[right=2pt] {\footnotesize $t_j = x_i$};
\draw (15.375,3.25)--(15.375,2.95);
\draw (15.375,3.1) circle (0.19cm);
\draw (15,2.5)--(15.75,2.5) node[right=2pt] {\footnotesize $t_j\neq x_i$};
\draw (15.375,2.5) circle (0.09cm);
\draw (15.375,2.5) circle (0.19cm);
\draw[thick] (15,2.0)--(15.75,2.0) node[right=2pt] {\footnotesize $s$};

\end{tikzpicture}
%} % END \fbox
} % END \resizebox
\end{center}
\caption{A first-degree spline (broken line) $s \in S^1_5[a,b]$ as approximation to the $\mu+2=19$ data $(x_0,f_0),\ldots,(x_{18},f_{18})$. $s$ has $k=5$ simple knots $t_1,\ldots,t_5$ and the boundary knots $t_0 := a := x_0$ and $t_6 := b := x_{18}$.}
\label{fig:fig1}
\end{figure}

%%%%%%%%%%%%%%%%%%%%%%%%%%%%%%%%%%%
%%%%%%%%%%%%%%%%%%%%%%%%%%%%%%%%%%%
% Section 2
%%%%%%%%%%%%%%%%%%%%%%%%%%%%%%%%%%%
%%%%%%%%%%%%%%%%%%%%%%%%%%%%%%%%%%%
\vspace{0.25in}
\setcounter{equation}{0}
\section{Existence theorem and other properties}
\label{sec:mainsection}
Our objective in this section is to demonstrate the existence of a best approximation in $S^1_k[a,b]$ as solution of Problem \ref{defApproxproblem}. It features only simple knots, is continuous, and has other properties important for the numerical calculation. We pave the way for the summarizing Theorem \ref{satzEigenschaftenBesteApprox} in several steps.

\begin{lemma}\label{minimalfolgenlemma} 
There is a minimizing sequence with bounded function values and derivatives on the interval $[a,b]$. More precisely: There is a constant $M > 0$ and a minimizing sequence $s^{(i)} \in S^1_k[a,b]$ with
\begin{equation*}
\left\|f-s^{(i)}\right\|_{p,X} \;\longrightarrow\; \inf_{s \in S^1_k[a,b]} \left\|f-s\right\|_{p,X}
\end{equation*}
for $i \rightarrow \infty$ where
\begin{equation}\label{eq:gleichungA} 
\left|s^{(i)}(x)\right| \leq M \;\;\text{and}\;\; \left|\frac{d}{dx} s^{(i)}(x)\right| \leq M \;\;\text{for all}\;\; x \in [a,b] \quad.
\end{equation}
Here, the upper estimates for the derivative in (\ref{eq:gleichungA}) in the knots of $s^{(i)}$ hold for the left- und right-hand derivatives.
\end{lemma}

\vspace{0.1in}
\begin{remark}\label{minimalfolgenbeispiel} 
The conclusions in Lemma \ref{minimalfolgenlemma} are not obvious as can be seen from the following example of a minimizing sequence: Despite boundedness in the $x_j$ (here: $-1,0,1$) the $s^{(i)}$ can have unbounded functions values and derivatives (in $t_2^{(i)} = 0.5$ in this example). The function
\begin{equation*}
s^{(i)}(x) \;:=\; \left\{
\begin{array}{lcl}
-x & , & x \leq -\frac{1}{i} \\
1 + (i-1)x & , & -\frac{1}{i} < x \leq \frac{1}{2} \\
i + (1-i)x & , & x > \frac{1}{2}
\end{array}
\right.
\end{equation*}
is from $S_2^1[-1,1]$ and has the function value $s^{(i)}\left(x_j\right) = 1$  in $x_0 := -1$, $x_1 := 0$ and $x_2 := 1$. But the sequence $\left(s^{(i)}\right)_{i \in \mathbb{N}}$ is not bounded on $[-1,1]$ since $s^{(i)}\left(t_2^{(i)}\right) = s^{(i)}\left(\frac{1}{2}\right) = \frac{1}{2}(i+1)$; see Fig. \ref{fig2}.
\end{remark}

% ===============================
% Figure 2
% ===============================
\begin{figure}[h]
\begin{center}
\resizebox{0.80\textwidth}{!}{
%\fbox{
\begin{tikzpicture} 
% x-Achse und y-Achse  
\draw[->,color=black,>=stealth] (-0.2,0) -- (12.5,0) node[below=2pt] {\normalsize $x$};
\draw[->,color=black,>=stealth] (0,-0.2) -- (0,7.2); 
% Einteilung und Beschriftung der x-Achse
\draw (0,0-.2)--(0,0) node[below=6pt]{\normalsize -1}; 
\draw (1.5,0-.1)--(1.5,0);
\draw (3,0-.2)--(3,0) node[below=6pt]{\normalsize -0.5};  
\draw (4.5,0-.1)--(4.5,0);
\draw (6,0-.2)--(6,0) node[below=6pt]{\normalsize 0};  
\draw (7.5,0-.1)--(7.5,0);
\draw (9,0-.2)--(9,0) node[below=6pt]{\normalsize 0.5};
\draw (10.5,0-.1)--(10.5,0); 
\draw (12,0-.2)--(12,0) node[below=6pt]{\normalsize 1};
% Einteilung und Beschriftung der y-Achse
\draw (0-.2,0)--(0,0) node[left=6pt]{\normalsize 0};
\draw (0-.1,0.31818)--(0,0.31818); 
\draw (0-.2,0.63636)--(0,0.63636) node[left=6pt]{\normalsize 0.5}; 
\draw (0-.1,0.95455)--(0,0.95455); 
\draw (0-.2,1.2727)--(0,1.2727) node[left=6pt]{\normalsize 1};  
\draw (0-.1,1.5909)--(0,1.5909);
\draw (0-.2,1.9091)--(0,1.9091) node[left=6pt]{\normalsize 1.5};
\draw (0-.1,2.2273)--(0,2.2273);
\draw (0-.2,2.5455)--(0,2.5455) node[left=6pt]{\normalsize 2};  
\draw (0-.1,2.8636)--(0,2.8636);
\draw (0-.2,3.1818)--(0,3.1818) node[left=6pt]{\normalsize 2.5};
\draw (0-.1,3.5)--(0,3.5);
\draw (0-.2,3.8182)--(0,3.8182) node[left=6pt]{\normalsize 3};  
\draw (0-.1,4.1364)--(0,4.1364);
\draw (0-.2,4.4545)--(0,4.4545) node[left=6pt]{\normalsize 3.5};
\draw (0-.1,4.7727)--(0,4.7727);
\draw (0-.2,5.0909)--(0,5.0909) node[left=6pt]{\normalsize 4};  
\draw (0-.1,5.4091)--(0,5.4091);
\draw (0-.2,5.7273)--(0,5.7273) node[left=6pt]{\normalsize 4.5};
\draw (0-.1,6.0455)--(0,6.0455);
\draw (0-.2,6.3636)--(0,6.3636) node[left=6pt]{\normalsize 5};  
\draw (0-.1,6.6818)--(0,6.6818);
\draw (0-.2,7)--(0,7) node[left=6pt]{\normalsize 5.5}; 
% Vorgaben (t_i,f_i)
\color{black} 
% Polygonzüge 
% Polygonzug 1
\draw[thin] (0,1.2727)--(3,0.63636)--(9,1.9091)--(12,1.2727); 
  
% Polygonzug 2
\draw[thick, dotted] (0,1.2727)--(4,0.42424)--(9,2.5455)--(12,1.2727);
  
% Polygonzug 3
\draw[thick,dashed] (0,1.2727)--(4.5,0.31818)--(9,3.1818)--(12,1.2727);
  
% Polygonzug 4
\draw[thick,densely dotted] (0,1.2727)--(4.8,0.25455)--(9,3.8182)--(12,1.2727);
  
% Polygonzug 5
\draw[very thick] (0,1.2727)--(5.4,0.12727)--(9,7)--(12,1.2727); 

% Legende
%\draw (12.75,-0.5)--(12.75,7.3); 
\draw[thin] (13,7)--(14,7) node[right=2pt] {\large $s^{(2)}$}; 
\draw[thick, dotted] (13,6.25)--(14,6.25) node[right=2pt] {\large $s^{(3)}$}; 
\draw[thick,dashed] (13,5.5)--(14,5.5) node[right=2pt] {\large $s^{(4)}$}; 
\draw[thick,densely dotted] (13,4.75)--(14,4.75) node[right=2pt] {\large $s^{(5)}$}; 
\draw[very thick] (13,4)--(14,4) node[right=2pt] {\large $s^{(10)}$};
  
\end{tikzpicture}
%} % END \fbox
} % END \resizebox
\end{center}
\caption{A few terms of the sequence $\left(s^{(i)}\right)_{i \in \mathbb{N}}$ from Remark \ref{minimalfolgenbeispiel}. The $s^{(i)}$ are bounded on a discrete set (here: $-1,0,1$) yet unbounded in between (here: $0.5$).}
\label{fig2}
\end{figure}
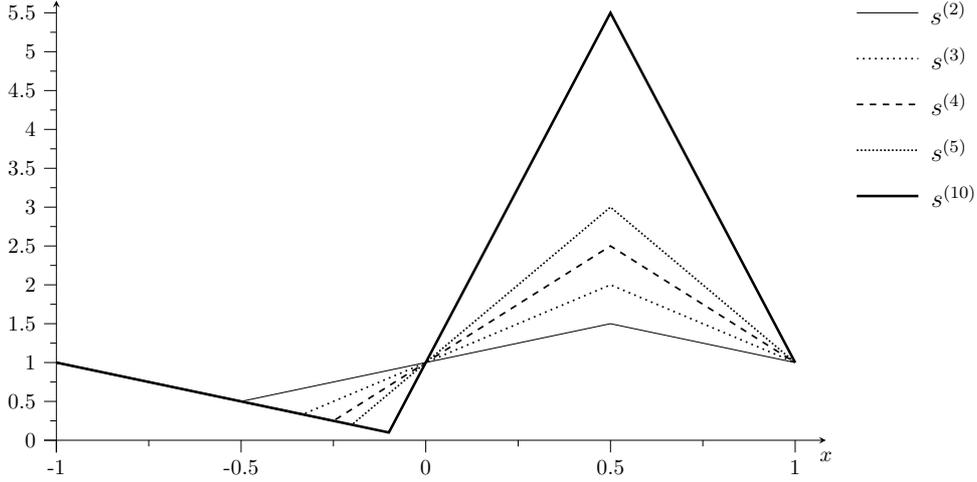

%%%%%%%%%%%%%%%%%%%%%%%%%%%%%%%
%%%%%%%%%%%%%%%%%%%%%%%%%%%%%%%
% Section 2 - Proof of Lemma 3
%%%%%%%%%%%%%%%%%%%%%%%%%%%%%%%
%%%%%%%%%%%%%%%%%%%%%%%%%%%%%%%
\begin{prooftheorem}
\textbf{of Lemma \ref{minimalfolgenlemma}:} Let $s^{(i)} \in S^1_k[a,b]$, $i \in \mathbb{N}$, be the terms of a minimizing sequence. Without loss of generality all knots of the $s^{(i)}$ are proper knots with discontinuous first derivatives. We will construct a new sequence with the claimed properties and with unnecessary oscillations ironed out. In this process function values in the data abscissae $s^{(i)}\left(x_j\right)$, $0 \leq j \leq \mu+1$, remain \textit{unchanged} such that the approximation power of the original $s^{(i)}$ on $X$ and thus the property of being a minimizing sequence is preserved. 

\vspace{0.05in}
Since $\left(s^{(i)}\right)_{i \in \mathbb{N}}$ is a minimizing sequence we have
\begin{equation}\label{eq:gleichungB} 
M' \;:=\; \sup_{\genfrac{}{}{0pt}{}{l \in \mathbb{N}}{0 \leq j \leq \mu+1}} \left|s^{(l)}\left(x_j\right)\right| \;<\; \infty \quad.
\end{equation}

We set
\begin{eqnarray} 
M_l'' &:=& \max_{1 \leq q \leq \mu+1} \left|\frac{s^{(l)}\left(x_q\right) - s^{(l)}\left(x_{q-1}\right)}{x_q-x_{q-1}}\right|, \label{eq:gleichungC} \\
M''' &:=& \sup_{l \in \mathbb{N}} M_l'',  \label{eq:gleichungD} \\
M'''' &:=& M' \;+\; (b-a) M''' \;\;\;\;\;\text{and}  \label{eq:gleichungE} \\
M &:=& \max\left(M''',M''''\right) \quad.  \label{eq:gleichungF}
\end{eqnarray}
We will show that $M'''$ is an upper bound for the absolute values of the derivatives and $M''''$ for the absolute values of a (suitably redefined, if necessary) minimizing sequence on $[a,b]$. (\ref{eq:gleichungB}) implies finiteness of the values (\ref{eq:gleichungC}) to (\ref{eq:gleichungF}), in particular $M < \infty$. We start the proof with the following
\begin{assertion}\label{behauptungAbleitungBeschraenkt} 
Without loss of generality the absolute values of the derivatives of the $s^{(i)}$ are bounded by $M_i''$:
\begin{equation}\label{eq:gleichungG} 
\left|\frac{d}{dx} s^{(i)}(x)\right| \;\leq\; M_i'' \;\;\;\text{for all}\;\;\; x \in [a,b] \quad.
\end{equation}
\end{assertion}
\begin{proof}
We prove Assertion \ref{behauptungAbleitungBeschraenkt} by induction for the interval $\left[x_0,x_q\right]$, $q = 1,2,\ldots$, $\mu+1$. If necessary, $s^{(i)}$ is changed such that estimate (\ref{eq:gleichungG}) holds for the interval $[x_0,x_q]$. The function values $s^{(i)}(x_j)$ remain unchanged and the property of being a minimizing sequence and (\ref{eq:gleichungB}) to (\ref{eq:gleichungF}) are untouched. For simplicity the possibly modified functions are denoted $s^{(i)}$ again.  

\vspace{0.05in}
For the \underline{induction basis}, this means the interval $\left[x_0,x_1\right]$, we have to check two cases:

\vspace{0.03in}
\underline{Case 1:} $s^{(i)}$ has at least one knot and thus $t_1^{(i)}\in\left(x_0,x_1\right)$. Then we move $t_1^{(i)}$ to $x_1$ and replace $s^{(i)}$ on $\left[x_0,x_1\right]$ by the straight line connecting the points $\left(x_0,s^{(i)}\left(x_0\right)\right)$ and $\left(x_1,s^{(i)}\left(x_1\right)\right)$. We denote the function thus defined by $s^{(i)}$ again for which the upper estimate from (\ref{eq:gleichungG}) now holds.

% ===============================
% Figure 3
% ===============================
\begin{figure}[h]
\begin{center}
%\fbox{
\begin{tikzpicture}
\draw (-2.5,0)--(3,0);
\draw (-2,0.1)--(-2,-0.1) node[below=4pt] {$x_0$};
\draw (0,0.1)--(0,-0.1) node[below=4pt] {$t_1^{(i)}$};
\draw (2.5,0.1)--(2.5,-0.1) node[below=4pt] {$x_1$};

\draw[fill] (2.5,2.1667) circle (1.5pt);

\draw[thick] (-2,1) to (0,0.5) to node[below=2pt] {$s_{\text{old}}^{(i)}$} (3,2.5);
\draw[thick,dashed] (-2,1) to node[above=2pt] {$s_{\text{new}}^{(i)}$} (2.5,2.1667);
\end{tikzpicture}
%} % END \fbox
\end{center}
\caption{Proof of Assertion \ref{behauptungAbleitungBeschraenkt}, induction basis, case 1.}
\label{fig3}
\end{figure}
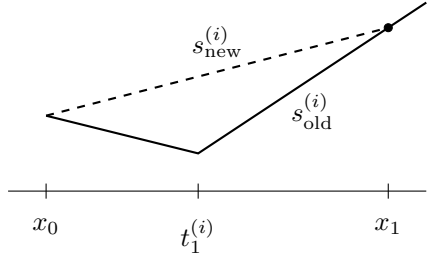

\vspace{0.1in}
\underline{Case 2:} $s^{(i)}$ has no knot in $\left(x_0,x_1\right)$. Therefore $s^{(i)}$ is a straight line on $\left[x_0,x_1\right]$ and the upper estimate (\ref{eq:gleichungG}) holds.

\vspace{0.1in}
For the \underline{induction step} let $s^{(i)}$ be a minimizing sequence for which estimate (\ref{eq:gleichungG}) holds on $\left[x_0,x_{q-1}\right]$. We show that a minimizing sequence exists for wich (\ref{eq:gleichungG}) holds on $\left[x_0,x_q\right]$. We have to consider several cases:

\vspace{0.03in}
\underline{Case 1:} $s^{(i)}$ does not have a knot in $\left(x_{q-1},x_q\right)$. Then $s^{(i)}$ is linear on $\left[x_{q-1},x_q\right]$ and (\ref{eq:gleichungG}) holds for all $x \in \left[x_0,x_q\right]$.

\vspace{0.03in}
\underline{Case 2:} $s^{(i)}$ has at least two knots $\left(x_{q-1},x_q\right)$. Replace $s^{(i)}$ on $\left[x_{q-1},x_q\right]$ by the straight line connecting $\left(x_{q-1},s^{(i)}\left(x_{q-1}\right)\right)$ and $\left(x_q,s^{(i)}\left(x_q\right)\right)$. We denote the function from $S_k^1[a,b]$ thus defined with $s^{(i)}$ again. The function values in all $x_j$ remain unchanged by this modification and (\ref{eq:gleichungG}) now holds for $s^{(i)}$ and all $x \in \left[x_0,x_q\right]$.

\vspace{0.06in}
\underline{Case 3:} $s^{(i)}$ has exactly one knot in $\left(x_{q-1},x_q\right)$.

\vspace{0.03in}
\underline{Case 3.1:} $x_{q-1}$ or $x_q$ are likewise knots of $s^{(i)}$. Then we replace $s^{(i)}$ on $\left[x_{q-1},x_q\right]$ by the straight line connecting $\left(x_{q-1},s^{(i)}\left(x_{q-1}\right)\right)$ and $\left(x_q,s^{(i)}\left(x_q\right)\right)$. The function thus constructed (denoted by $s^{(i)}$ again) meets estimate (\ref{eq:gleichungG}) on $\left[x_0,x_q\right]$.

\vspace{0.03in}
\underline{Case 3.2:} The knot $t_j^{(i)} \in \left(x_{q-1},x_q\right)$ is the only knot of $s^{(i)}$ in the interval $\left[x_{q-1},x_q\right]$.

\vspace{0.03in}
We define the straight line connecting the points $\left(x_{q-1},s^{(i)}\left(x_{q-1}\right)\right)$ and $\left(x_q,s^{(i)}\left(x_q\right)\right)$
\begin{equation*}
\sigma(x) \;:=\; s^{(i)}\left(x_{q-1}\right) + \left(x - x_{q-1}\right) \frac{s^{(i)}\left(x_q\right)-s^{(i)}\left(x_{q-1}\right)}{x_q - x_{q-1}}
\end{equation*}
and continue the distinction of cases with

\vspace{0.03in}
\underline{Case 3.2.1:} In this case $s^{(i)}\left(t_j^{(i)}\right) = \sigma\left(t_j^{(i)}\right)$. Then $s^{(i)}$ is a straight line on $\big[x_{q-1}-\varepsilon,x_q\big]$ for sufficiently small $\varepsilon > 0$ and the induction hypothesis extends from $\big[x_0,x_{q-1}\big]$ to $\big[x_0,x_q\big]$. 

\vspace{0.03in}
\underline{Case 3.2.2:} In this case $s^{(i)}\left(t_j^{(i)}\right) < \sigma\left(t_j^{(i)}\right)$.

\vspace{0.03in}
\underline{Case 3.2.2.1:} Subcase $q \leq \mu$.

\vspace{0.03in}
\underline{Case 3.2.2.1.1:} We consider the case $s^{(i)}\left(x_{q+1}\right) \leq \sigma\left(x_{q+1}\right)$. Since $s^{(i)}\left(t_j^{(i)}\right)$ lies strictly below the straight line $\sigma$, there must be a knot $t_{j+1}^{(i)} \in \left(x_q,x_{q+1}\right)$ in the following section. Furthermore, there is a $x \in \left(x_q,x_{q+1}\right]$ such that the point $\left(x,s^{(i)}(x)\right)$ also lies on the straight line $\sigma$. We replace $s^{(i)}$ on $\left[x_{q-1},x\right]$ by $\sigma$. Thus, the knots $t_j^{(i)}$ and $t_{j+1}^{(i)}$ vanish whereas $x_{q-1}$ and $x$ are new knots. We denote the function thus constructed by $s^{(i)}$ again. (\ref{eq:gleichungG}) now holds for $s^{(i)}$ on $\left[x_{q-1},x_q\right]$ by construction and because of the induction hypothesis therefore on $[x_0,x_q]$.

% =========================
% Figure 4
% =========================
\begin{figure}[h]
\begin{center}
%\fbox{
\begin{tikzpicture}
\draw (-2.5,0)--(3.5,0);
\draw (-2,0.1)--(-2,-0.1) node[below=2pt] {\small $x_{q-1}$};
\draw (-1,0.1)--(-1,-0.1) node[below=2pt] {\small $t_j^{(i)}$};
\draw (0,0.1)--(0,-0.1) node[below=2pt] {\small $x_q$};
\draw (1,0.1)--(1,-0.1) node[below=2pt] {\small $t_{j+1}^{(i)}$};
\draw (2,0.1)--(2,-0.1) node[below=2pt] {\small $x$};
\draw (3,0.1)--(3,-0.1) node[below=2pt] {\small $x_{q+1}$};

\draw[->,>=stealth] (-1.25,-0.35)--(-1.55,-0.35);
\draw[->,>=stealth] (1.25,-0.35)--(1.75,-0.35);

\draw[fill] (0,1.5) circle (1.5pt);
\draw[fill] (-2,1.5) circle (1.5pt);
\draw[fill] (3,0.5) circle (1.5pt);

\draw[thick] (-2.5,2)--(-1,0.5)--(1,2.5)--(3.25,0.25);
\draw[thick,dashed] (2,1.5)--(-2,1.5);

%\draw (3.75,-0.6)--(3.75,2.6); 
\draw (4,2.3)--(5,2.3) node[right=2pt] {\large $s^{(i)}_{\text{old}}$}; 
\draw[dashed] (4,1.55)--(5,1.55) node[right=2pt] {\large $s^{(i)}_{\text{new}}$}; 
\end{tikzpicture}
%} % END \fbox
\end{center}
\caption{Illustration of the proof of Assertion \ref{behauptungAbleitungBeschraenkt}, case 3.2.2.1.1.}
\label{fig4}
\end{figure}
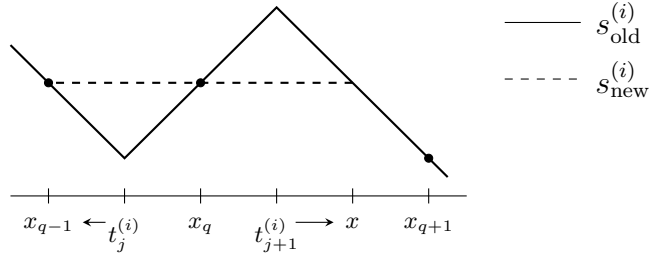

\vspace{0.1in}
\underline{Case 3.2.2.1.2:} In this case $s^{(i)}\left(x_{q+1}\right) > \sigma\left(x_{q+1}\right)$.

\vspace{0.03in}
\underline{Case 3.2.2.1.2.1:} $s^{(i)}$ has exactly one (proper) knot $t_{j+1}^{(i)}$ in $\left(x_q,x_{q+1}\right]$. We denote whith $\varphi$ the straight line with function value $s^{(i)}\left(x_q\right)$ and derivative $\frac{d}{dx}s^{(i)}\left(x_q\right)$ in $x_q$.

\vspace{0.03in}
\underline{Case 3.2.2.1.2.1.1:} In this case $s^{(i)}\left(x_{q+1}\right) < \varphi(x_{q+1})$. Let $x \in \left(x_{q-1},t_j^{(i)}\right)$ be the first (smallest) point of intersection of $s^{(i)}$ with the straight line running through the points $\left(x_q,s^{(i)}\left(x_q\right)\right)$ and $\left(x_{q+1},s^{(i)}\left(x_{q+1}\right)\right)$. We replace $s^{(i)}$ on $\left[x,x_{q+1}\right]$ by this straight line. Thus,  $x$ and $x_{q+1}$ become new knots while $t_j^{(i)}$ and $t_{j+1}^{(i)}$ are no longer knots. The function thus defined is denoted by $s^{(i)}$ again. Estimate (\ref{eq:gleichungG}) holds for $s^{(i)}$ on $\left[x_0,x_{q-1}\right]$ because of the induction hypothesis and on $\left[x_{q-1},x\right]$ because the derivative of $s^{(i)}$ in this interval equals the left-hand derivative of $s^{(i)}$ in $x_{q-1}$ (because of the induction hypothesis again). Finally, (\ref{eq:gleichungG}) holds on $\left[x,x_q\right]$ because $\frac{d}{dx}s^{(i)}(z) = \frac{s^{(i)}\left(x_{q+1}\right) - s^{(i)}\left(x_q\right)}{x_{q+1}-x_q}$ for all $z \in \left[x,x_q\right]$ which follows from the construction of $s^{(i)}$.

% =======================
% Figure 5
% =======================
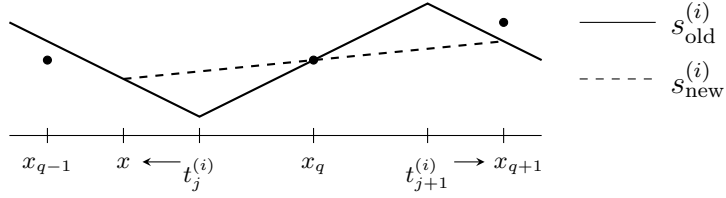
\begin{figure}[h]
\begin{center}
%\fbox{
\begin{tikzpicture}
\draw (-4,0)--(3,0);
\draw (-3.5,0.1)--(-3.5,-0.1) node[below=2pt] {\small $x_{q-1}$};
\draw (-2.5,0.1)--(-2.5,-0.1) node[below=2pt] {\small $x$};
\draw (-1.5,0.1)--(-1.5,-0.1) node[below=2pt] {\small $t_j^{(i)}$};
\draw (0,0.1)--(0,-0.1) node[below=2pt] {\small $x_q$};
\draw (1.5,0.1)--(1.5,-0.1) node[below=2pt] {\small $t_{j+1}^{(i)}$};
\draw (2.5,0.1)--(2.5,-0.1);
\draw (2.7,-0.1) node[below=2pt] {\small $x_{q+1}$};

\draw[->,>=stealth] (-1.75,-0.35)--(-2.25,-0.35);
\draw[->,>=stealth] (1.85,-0.35)--(2.25,-0.35);

\draw[fill] (0,1) circle (1.5pt);
\draw[fill] (-3.5,1.0) circle (1.5pt);
\draw[fill] (2.5,1.5) circle (1.5pt);

\draw[thick] (-4,1.5)--(-1.5,0.25)--(1.5,1.75)--(3,1.0);
\draw[thick,dashed] (-2.5,0.75)--(2.5,1.25);

%\draw (3.25,-1)--(3.25,1.9); 
\draw (3.5,1.5)--(4.5,1.5) node[right=2pt] {\large $s^{(i)}_{\text{old}}$}; 
\draw[dashed] (3.5,0.75)--(4.5,0.75) node[right=2pt] {\large $s^{(i)}_{\text{new}}$}; 
\end{tikzpicture}
%} % END \fbox
\end{center}
\caption{Illustration of the proof of Assertion \ref{behauptungAbleitungBeschraenkt}, case 3.2.2.1.2.1.1}
\label{fig5}
\end{figure}

\vspace{0.2in}
\underline{Case 3.2.2.1.2.1.2:} In this case $s^{(i)}\left(x_{q+1}\right) = \varphi(x_{q+1})$. Then $s^{(i)}$ is a straight line on $\big[x_q,x_{q+1}\big]$ and the only proper knot $t^{(i)}_{j+1}$ in $\big(x_q,x_{q+1}\big]$ must coincide with $x_{q+1}$. Estimate (\ref{eq:gleichungG}) holds even on $\big[x_0,x_{q+1}\big]$.

\vspace{0.03in}
\underline{Case 3.2.2.1.2.1.3:} Here we consider the  case $s^{(i)}\left(x_{q+1}\right) > \varphi(x_{q+1})$. Then there exists a point $x \in \left(x_q,t_{j+1}^{(i)}\right)$ as intersection of the straight line $\sigma$ connecting the points $\left(x_{q-1},s^{(i)}\left(x_{q-1}\right)\right)$ and $\left(x_q,s^{(i)}\left(x_q\right)\right)$ with the straight line connecting the points $\left(t_{j+1}^{(i)},s^{(i)}\left(t_{j+1}^{(i)}\right)\right)$ and $\left(x_{q+1},s^{(i)}\left(x_{q+1}\right)\right)$. We replace $s^{(i)}$ on $\left(x_{q-1},x\right)$ by $\sigma$ and on $\left[x,x_{q+1}\right]$ by the second named straight line. The absolute values of the derivatives of the thus defined (new) broken line $s^{(i)}$ are again bounded by $M_i''$.

% ================
% Figure 6
% ================
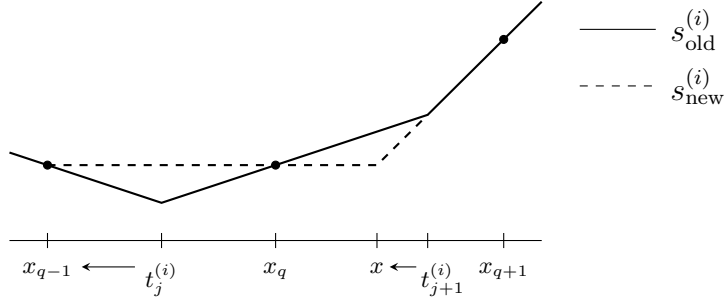
\begin{figure}[h]
\begin{center}
%\fbox{
\begin{tikzpicture}
\draw (-3,0)--(4,0);
\draw (-2.5,0.1)--(-2.5,-0.1) node[below=2pt] {\small $x_{q-1}$};
\draw (-1,0.1)--(-1,-0.1) node[below=2pt] {\small $t_j^{(i)}$};
\draw (0.5,0.1)--(0.5,-0.1) node[below=2pt] {\small $x_q$};
\draw (1.8333,0.1)--(1.8333,-0.1) node[below=2pt] {\small $x$};
\draw (2.5,0.1)--(2.5,-0.1);
\draw (2.7,-0.1) node[below=2pt] {\small $t_{j+1}^{(i)}$};
\draw (3.5,0.1)--(3.5,-0.1) node[below=2pt] {\small $x_{q+1}$};

\draw[->,>=stealth] (-1.35,-0.35)--(-2.05,-0.35);
\draw[->,>=stealth] (2.35,-0.35)--(2,-0.35);

\draw[fill] (0.5,1) circle (1.5pt);
\draw[fill] (-2.5,1) circle (1.5pt);
\draw[fill] (3.5,2.6667) circle (1.5pt);

\draw[thick] (-3,1.1667)--(-1,0.5)--(2.5,1.6667)--(4,3.1667);
\draw[thick,dashed] (-2.5,1)--(1.8333,1)--(2.5,1.6667);

%\draw (4.25,-0.75)--(4.25,3.3); 
\draw (4.5,2.8)--(5.5,2.8) node[right=2pt] {\large $s^{(i)}_{\text{old}}$}; 
\draw[dashed] (4.5,2.05)--(5.5,2.05) node[right=2pt] {\large $s^{(i)}_{\text{new}}$}; 
\end{tikzpicture}
%} % END \fbox
\end{center}
\caption{Illustration of the proof of Assertion \ref{behauptungAbleitungBeschraenkt}, case 3.2.2.1.2.1.3.}
\label{fig6}
\end{figure}

\vspace{0.03in}
\underline{Case 3.2.2.1.2.2:} $s^{(i)}$ has no (proper) knot in $\left(x_q,x_{q+1}\right]$. Then $t_j^{(i)}$ is the only knot of $s^{(i)}$ in $\left[x_{q-1},x_{q+1}\right]$. (\ref{eq:gleichungG}) holds on the interval $\left[x_0,x_{q-1}\right]$ because of the induction hypothesis. (\ref{eq:gleichungG}) also holds on $\left[x_{q-1},t_j^{(i)}\right]$ because of the induction hypothesis since for $z \in \left[x_{q-1},t_j^{(i)}\right]$ we have: $\frac{d}{dx}s^{(i)}(z) = \frac{d}{dx}s^{(i)}\left(x_{q-1}^{-}\right)$ where the latter denotes the left-side derivative of \, $s^{(i)}$ in $x_{q-1}$. \; Finally, (\ref{eq:gleichungG}) ist valid on \, $\left[t_j^{(i)},x_q\right]$ \, because $\frac{d}{dx}s^{(i)}(z) = \frac{s^{(i)}\left(x_{q+1}\right) - s^{(i)}\left(x_q\right)}{x_{q+1}-x_q}$ holds for all $z \in \left[t_j^{(i)},x_q\right]$.

\vspace{0.03in}
\underline{Case 3.2.2.1.2.3:} There are at least two (proper) knots $t_{j+1}^{(i)},t_{j+2}^{(i)}$ in the intervall $\big(x_q,x_{q+1}\big]$. Then we replace $s^{(i)}$ on $\big[x_{q-1},x_q\big]$ by the straight line connecting the points $\big(x_{q-1},s^{(i)}(x_{q-1})\big)$ and $\big(x_q,s^{(i)}(x_q)\big)$ and on $\big[x_q,x_{q+1}\big]$ by the straight line connecting $\big(x_q,s^{(i)}(x_q)\big)$ and $\big(x_{q+1},s^{(i)}(x_{q+1})\big)$. The knots $t_j^{(i)}, t_{j+1}^{(i)}$ and $t_{j+2}^{(i)}$ thus disappear and  $x_{q-1},x_q,x_{q+1}$ become knots by this modification. The thus modified $s^{(i)}$ fulfills (\ref{eq:gleichungG}) on $\big[x_0,x_{q-1}\big]$ because of the induction hypothesis and on $\big[x_{q-1},x_{q+1}\big]$ by construction.

\vspace{0.03in}
\underline{Case 3.2.2.2:} Now we consider the case $q = \mu+1$. Then we move the knot in $\big(x_{q-1},x_q\big)$ onto $x_{q-1}$ and replace $s^{(i)}$ on the interval $\left[x_{q-1},x_q\right]$ by the straight line connecting the points $\left(x_{q-1},s^{(i)}\left(x_{q-1}\right)\right)$ and $\left(x_q,s^{(i)}\left(x_q\right)\right)$. The estimate (\ref{eq:gleichungG}) now holds for the modified $s^{(i)}$ on $\big[x_0,x_{\mu+1}\big]$.

\vspace{0.03in}
\underline{Case 3.2.3:} We consider the case $s^{(i)}\left(t_j^{(i)}\right) > \sigma\left(t_j^{(i)}\right)$. This case is treated in analogy to case 3.2.2 an its subcases.

\vspace{0.03in}
This proves Assertion \ref{behauptungAbleitungBeschraenkt}.
\end{proof}

\vspace{0.1in}
We can now conclude the proof of Lemma \ref{minimalfolgenlemma}: For terms $s^{(i)}$ of a minimizing sequence the $\big|s^{(i)}\big|$ in (\ref{eq:gleichungB}) are bounded with respect to $i\in\mathbb{N}$ and $0 \leq j \leq \mu+1$ and therefore also the $M_i''$ in (\ref{eq:gleichungC}) by $M'''\in\mathbb{R}$. 

\vspace{0.05in}
For a minimizing sequence $s^{(i)}$ (modified as in the proof of Assertion \ref{behauptungAbleitungBeschraenkt}, if necessary) the derivatives are therefore bounded by $M'''$ on $[a,b]$ and the function values by $M''''$. Thus, $M := \max(M''',M'''')$ bounds function values and derivatives of the $s^{(i)}$ on all of $[a,b]$. \hfill $\square$
\end{prooftheorem}

%%%%%%%%%%%%%%%%%%%%%%%%%%%%%%%%%%%%
%%%%%%%%%%%%%%%%%%%%%%%%%%%%%%%%%%%%
% Section 2 - Theorem 6 to Remark 11
%%%%%%%%%%%%%%%%%%%%%%%%%%%%%%%%%%%%
%%%%%%%%%%%%%%%%%%%%%%%%%%%%%%%%%%%%

\vspace{0.2in}
We now turn to the promised existence theorem:
\begin{theorem}\label{existenzsatz} 
With $1 \leq p \leq \infty$ there is a  $s^{\ast} \in S^1_k[a,b]$ such that
\begin{equation*}
\left\|f-s^{\ast}\right\|_{p,X} \;\;=\;\; \inf_{s \in S^1_k[a,b]} \left\|f-s\right\|_{p,X} \quad.
\end{equation*}
\end{theorem}
\begin{proof}
Let $s^{(i)} \in S^1_k[a,b]$ be a bounded minimizing sequence according to Lemma \ref{minimalfolgenlemma} with $M$ as bound. The following additional claims can be fulfilled by transition to a subsequence if necessary: For fixed $r$ with $0 \leq r \leq k$ each $s^{(i)}$ has $r$ (proper) knots $t_1^{(i)},\ldots,t_r^{(i)}$ with
\begin{equation}\label{eq:gleichungH} 
a =: t_0^{(i)} < t_1^{(i)} < \ldots < t_r^{(i)} < t_{r+1}^{(i)} := b \quad.
\end{equation}
Furthermore, let all sequences of knots $\left(t_j^{(i)}\right)_{i \in \mathbb{N}}$ be convergent and we denote by $\tau_1,\ldots,\tau_{r'}$ the distinct limits in the open interval $(a,b)$ where $r' \leq r$. In addition, $\tau_0 := a$ and $\tau_{r'+1} := b$ can also be limits of these sequences. Then holds
\begin{equation}\label{eq:gleichungI} 
a =: \tau_0 < \tau_1 <  \ldots < \tau_{r'} < \tau_{r'+1} := b \quad.
\end{equation}
With
\begin{equation}\label{eq:gleichungJ} 
\tau_{l,l+1} := \frac{1}{2}\left(\tau_l + \tau_{l+1}\right), \;\;\;\;0 \leq l \leq r'
\end{equation}
let the sequences 
\begin{equation}\label{eq:gleichungK} 
\left( s^{(i)}\left(\tau_{l,l+1}\right)\right)_{i \in \mathbb{N}} \hspace{0.3in}\text{and}\hspace{0.3in} \left( \frac{d}{dx}s^{(i)}\left(\tau_{l,l+1}\right)\right)_{i \in \mathbb{N}}
\end{equation}
(bounded because of Lemma \ref{minimalfolgenlemma}) also be convergent for each $l$ with $0 \leq l \leq r'$.

\vspace{0.1in}
We can now progress to the definition of $s^{\ast}$: For $x \in \left[\tau_l,\tau_{l+1}\right)$, $0 \leq l \leq r'-1$ or $x \in \left[\tau_{r'},\tau_{r'+1}\right]$ we set
\begin{equation}\label{eq:gleichungL} 
s^{\ast}(x) \;:=\; \lim_{i \rightarrow \infty} s^{(i)}\left(\tau_{l,l+1}\right) + \left(x-\tau_{l,l+1}\right) \cdot \lim_{i \rightarrow \infty} \frac{d}{dx} s^{(i)}\left(\tau_{l,l+1}\right) \quad.
\end{equation}
By definition $s^{\ast}: [a,b] \rightarrow \mathbb{R}$ is linear on $\left[\tau_l,\tau_{l+1}\right)$ for $0 \leq l \leq r'-1$ and on $\left[\tau_{r'},\tau_{r'+1}\right]$.

\vspace{0.1in}
We prove successively:
\begin{assertion}\label{behauptungGleichmaessigeKonvergenz} 
The sequence $\left(s^{(i)}\right)_{i \in \mathbb{N}}$ converges uniformly to $s^{\ast}$ on $[a,b]$:
\begin{equation*}
\forall\,\varepsilon> 0 \; \exists\,i_0\in\mathbb{N}\;:\;\; \forall\, i\geq i_0,\; x\in[a,b]\;:\;\; \big|s^{(i)}(x)-s^\ast(x)\big| < \varepsilon \quad.
\end{equation*}
\end{assertion}
\begin{assertion}\label{behauptungLinearUndStetig} 
$s^{\ast} \in S^1_k[a,b]$ is piecewiese linear and continuous on $[a,b]$.
\end{assertion}
\begin{assertion}\label{behauptungKonvergenzZusammenfassung} 
The sequence $\left(s^{(i)}\right)_{i \in \mathbb{N}}$ converges uniformly to the continuous, piecewiese linear function $s^{\ast}$.
\end{assertion}
\begin{assertion}\label{behauptungBesteApproximation} 
$s^{\ast}$ is a best approximation to $f$ on $X$.
\end{assertion}

\vspace{0.2in}
Proof of Assertion \ref{behauptungGleichmaessigeKonvergenz}: Let all knots chosen in this proof be proper knots and $\varepsilon > 0$ be given arbitrarily small. Let $i_0 \in \mathbb{N}$ be chosen sufficiently great such that for all $i \geq i_0$:
\begin{equation} 
\left|s^{(i)}\left(\tau_{l,l+1}\right) - s^{\ast}\left(\tau_{l,l+1}\right) \right| \;<\; \frac{\varepsilon}{4} \;\;\;\text{for $0 \leq l \leq r'$}, \label{eq:gleichungM} 
\end{equation}
\begin{equation}
\left|\frac{d}{dx}s^{(i)}\left(\tau_{l,l+1}\right) - \frac{d}{dx}s^{\ast}\left(\tau_{l,l+1}\right) \right| \;<\; \frac{\varepsilon}{4(b-a)} \;\;\;\text{for $0 \leq l \leq r'$}, \label{eq:gleichungN} 
\end{equation}
\begin{equation}
\left|t_l^{(i)} - \lim_{i \rightarrow \infty} t_l^{(i)}\right| \;<\; \frac{\varepsilon}{8 M} \;\;\;\text{for $l=1,2,\ldots,r$}, \label{eq:gleichungO}
\end{equation}
\begin{equation}
\left|t_l^{(i)} - \lim_{i \rightarrow \infty} t_l^{(i)}\right| \;<\; \frac{1}{2} \cdot \min_{q=1,2,\ldots,r'+1} \left(\tau_q - \tau_{q-1}\right) \;\;\text{for $l=1,2,\ldots,r$}. \label{eq:gleichungP}
\end{equation}
For $i\geq i_0$ and $x \in \left[a,b\right]$ we have $x\in\big[\tau_j,\tau_{j+1}\big)$ for a $j$ with $0\leq j\leq r'$ (cases 1 and 2) or $x=b$ (case 3). We have to show: $\left|s^{(i)}(x) - s^{\ast}(x)\right| < \varepsilon$.

\vspace{0.03in}
\underline{Case 1:} $x \in \left[\tau_{j,j+1},\tau_{j+1}\right)$ holds, that is $x$ lies in the right-hand half of $\left[\tau_j,\tau_{j+1}\right)$.

\vspace{0.03in}
\underline{Case 1.1:} $s^{(i)}$ has no knot in $\left[\tau_{j,j+1},x\right)$ and is therefore linear on this subintervall and from (\ref{eq:gleichungM}) and (\ref{eq:gleichungN}) thus follows:
\begin{eqnarray*}
&& \left|s^{(i)}(x)-s^{\ast}(x)\right| \\
&=& \left|s^{(i)}\left(\tau_{j,j+1}\right) + \left(x-\tau_{j,j+1}\right) \cdot \frac{d}{dx}s^{(i)}\left(\tau_{j,j+1}\right)\right. \\
&& \;\; - \left.\left[s^{\ast}\left(\tau_{j,j+1}\right) + \left(x-\tau_{j,j+1}\right) \cdot \frac{d}{dx}s^{\ast}\left(\tau_{j,j+1}\right)\right]\right| \\
&\leq& \left|s^{(i)}\left(\tau_{j,j+1}\right) - s^{\ast}\left(\tau_{j,j+1}\right)\right| \;+\; (b-a) \left|\frac{d}{dx}s^{(i)}\left(\tau_{j,j+1}\right) - \frac{d}{dx} s^{\ast}\left(\tau_{j,j+1}\right)\right| \\
&\leq& \frac{\varepsilon}{4} \;+\; \frac{\varepsilon}{4} \;\;<\;\; \varepsilon \quad.
\end{eqnarray*}

\vspace{0.03in}
\underline{Case 1.2:} $s^{(i)}$ has a knot in $\left[\tau_{j,j+1},x\right)$. Let $t_l^{(i)}$ be the smallest knot of $s^{(i)}$ in the interval $\left[\tau_{j,j+1},x\right)$. Then follows $\tau_{j,j+1} < t_l^{(i)} < x$ from (\ref{eq:gleichungP}) and we get:
\begin{eqnarray*}
&& \left|s^{(i)}(x)-s^{\ast}(x)\right| \\
&\leq& \left|s^{(i)}(x) - s^{(i)}\left(t_l^{(i)}\right)\right| + \left|s^{(i)}\left(t_l^{(i)}\right) - s^{\ast}\left(t_l^{(i)}\right)\right| + \left|s^{\ast}\left(t_l^{(i)}\right) - s^{\ast}(x)\right| \\
&\leq& M \left|t_l^{(i)} - x\right| + \frac{\varepsilon}{4} + \frac{\varepsilon}{4} + 2 M \left|t_l^{(i)} - x\right| \\
&\leq& M \left|t_l^{(i)} - \tau_{j+1}\right| \;+\; \frac{2\varepsilon}{4} \;+\; 2 M \left|t_l^{(i)} - \tau_{j+1}\right| \\
&\leq& \frac{\varepsilon}{8} + \frac{2\varepsilon}{4} + \frac{\varepsilon}{4} \;\;<\;\; \varepsilon \quad.
\end{eqnarray*}

% ======================
% Figure 7
% ======================
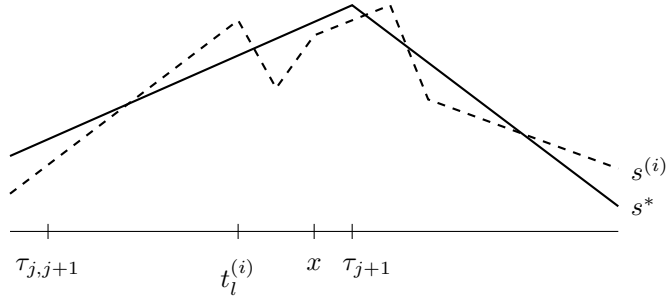
\begin{figure}[h]
\begin{center}
%\fbox{
\begin{tikzpicture}
\draw (-4,0)--(4,0);
\draw (-3.5,0.1)--(-3.5,-0.1) node[below=4pt] {$\tau_{j,j+1}$};
\draw (-1,0.1)--(-1,-0.1) node[below=4pt] {$t_l^{(i)}$};
\draw (0,0.1)--(0,-0.1) node[below=4pt] {$x$};
\draw (0.5,0.1)--(0.5,-0.1);
\draw (0.7,-0.1) node[below=4pt] {$\tau_{j+1}$};

\draw[thick] (-4,1)--(0.5,3)--(4,0.3333) node[right=1pt] {$s^\ast$};
\draw[thick,dashed] (-4,0.5)--(-1,2.8)--(-0.5,1.9)--(0,2.6)--(1,3)--(1.5,1.75)--(4,0.8333) node[right=1pt] {$s^{(i)}$};
\end{tikzpicture}
%} % END \fbox
\end{center}
\caption{Illustration of the proof of Assertion \ref{behauptungGleichmaessigeKonvergenz}, case 1.2.}
\label{fig7}
\end{figure}

\vspace{0.15in}
Here, $M$ (see (\ref{eq:gleichungF})) is an upper bound for the derivatives of $s^{(i)}$, the upper estimate
\begin{equation*}
\left|s^{(i)}\left(t_l^{(i)}\right) - s^{\ast}\left(t_l^{(i)}\right)\right| \leq \frac{\varepsilon}{4} + \frac{\varepsilon}{4}
\end{equation*}
follows in analogy to case 1.1 from (\ref{eq:gleichungM}) and (\ref{eq:gleichungN}) and $2M$ is an an upper bound for the absolute values of the derivatives of $s^{\ast}$. In addition, the last line of the above sequence of inequalities follows from the fact that because of (\ref{eq:gleichungP}) the sequence $\left(t_l^{(q)}\right)_{q \in \mathbb{N}}$ converges to $\tau_{j+1}$ which allows for an application of (\ref{eq:gleichungO}).

\vspace{0.03in}
\underline{Case 2:} $x$ lies in the left half of $\left[\tau_j,\tau_{j+1}\right)$: $x \in \left[\tau_j,\tau_{j,j+1}\right)$. This case ist treated analogously to case 1.

\vspace{0.03in}
\underline{Case 3:} In this case $x = \tau_{r'+1} = b$. The proof can again be done in analogy to case 1, which completes the proof of Assertion \ref{behauptungGleichmaessigeKonvergenz}.

\vspace{0.1in}
Next, we prove Assertion \ref{behauptungLinearUndStetig}: Since $s^{\ast}$ is piecewise linear by construction it suffices to show that $s^{\ast}$ is continuous in the knots $\tau_1,\ldots,\tau_{r'}$. Let $j$ be arbitrarily but fixed $(1 \leq j \leq r')$ and $\varepsilon > 0$ given. Because of Assertion \ref{behauptungGleichmaessigeKonvergenz} there is a $i_0 \in \mathbb{N}$ with
\begin{equation}\label{eq:gleichungQ}
\left|s^{(i)}(z) - s^{\ast}(z)\right| \;<\; \frac{\varepsilon}{3} \;\;\;\;\;\; \forall \, z \in \left[a,b\right]\;,\;\; i \geq i_0 \quad.
\end{equation}
Furthermore, for $x \in \left(\tau_{j-1,j},\tau_j\right)$ and $y \in \left[\tau_j,\tau_{j,j+1}\right)$ let the following upper estimate hold:
\begin{equation}\label{eq:gleichungR} 
|x - y| \;<\; \frac{\varepsilon}{3M} \;=:\; \delta \quad.
\end{equation}
From that follows
\begin{eqnarray*}
&& \left|s^{\ast}(x)-s^{\ast}(y)\right| \\
&\leq& \left|s^{\ast}(x) - s^{(i)}\left(x\right)\right| + \left|s^{(i)}\left(x\right) - s^{(i)}\left(y\right)\right| + \left|s^{(i)}\left(y\right) - s^{\ast}(y)\right| \\
&<& \frac{\varepsilon}{3} \;+\; M \left|x-y\right| \;+\; \frac{\varepsilon}{3} \\
&<& \frac{\varepsilon}{3} \;+\; \frac{\varepsilon}{3} \;+\; \frac{\varepsilon}{3} \;\;\;=\;\;\; \varepsilon,
\end{eqnarray*}
where the second estimate follows from (\ref{eq:gleichungQ}) and the boundedness of the derivatives of  $s^{(i)}$, while the last estimate holds because of (\ref{eq:gleichungR}). This proves Assertion \ref{behauptungLinearUndStetig}.

\vspace{0.1in}
Assertion \ref{behauptungKonvergenzZusammenfassung} summarizes Assertion \ref{behauptungGleichmaessigeKonvergenz} and \ref{behauptungLinearUndStetig}. The proof of Theorem \ref{existenzsatz} is concluded with the proof of Assertion \ref{behauptungBesteApproximation}: Because the $s^{(i)}$ converge uniformly to $s^{\ast}$ (see Assertion \ref{behauptungKonvergenzZusammenfassung}) we conclude
\begin{equation*}
\lim_{i\rightarrow\infty}s^{(i)}\left(x_j\right) \;=\; s^{\ast}\left(x_j\right), \;\; j =0,1,\ldots,\mu+1.
\end{equation*}
Since without loss of generality $\left\|f-s^{(i)}\right\|_{p,X}$ converges to 
\begin{equation*}
\inf_{s \in S^1_k[a,b]} \left\|f-s\right\|_{p,X}
\end{equation*}
we get for $p < \infty$
\begin{eqnarray*}
\inf_{s \in S^1_k[a,b]} \left\|f-s\right\|_{p,X} &=& \lim_{i \rightarrow \infty} \left\|f - s^{(i)}\right\|_{p,X} \;\;=\;\; \lim_{i \rightarrow \infty} \left( \sum_{j=0}^{\mu+1} \left|f_j - s^{(i)}\left(x_j\right)\right|^p\right)^{\frac{1}{p}} \\
&=& \left( \sum_{j=0}^{\mu+1} \left|f_j - s^{\ast}\left(x_j\right)\right|^p\right)^{\frac{1}{p}} \;\;=\;\; \left\|f-s^{\ast}\right\|_{p,X} \quad.
\end{eqnarray*}
The case $p=\infty$ is treated in analogy.

\vspace{0.05in}
In case $r < k$ ($s^\ast$ has less than $k$ knots) these can be completed by additional $(k-r)$ improper knots, such that $s^\ast$ is seen to be an element of $S^1_k[a,b]$ in this case, too.

\vspace{0.05in}
This winds up the proof of Assertion \ref{behauptungBesteApproximation} and the proof of Theorem \ref{existenzsatz} is thus completed. \hfill $\square$
\end{proof}

\vspace{0.2in}
\begin{remark}\label{bemerkungBzglArzelaAscoli}
Lemma \ref{minimalfolgenlemma} also guarantees the uniform boundedness and equicontinuity of a (subsequence of a) minimizing sequence. The Theorem of Arzel\`{a}-Ascoli therefore also implies that a minimizing sequence possesses a subsequence converging uniformly to a continuous limit function.
\end{remark}

%%%%%%%%%%%%%%%%%%%%%%%%%%%%%%%%
%%%%%%%%%%%%%%%%%%%%%%%%%%%%%%%%
% Section 2 - Theorem 12 to end
%%%%%%%%%%%%%%%%%%%%%%%%%%%%%%%%
%%%%%%%%%%%%%%%%%%%%%%%%%%%%%%%%

\vspace{0.2in}
For numerical prodedures it is advantageous to know more characteristic features of a best approximation. In this connection it is helpful to distinguish between knots coinciding with a data abscissa and knots situated between data abscissae. We call a knot $t_j$ of a first order spline $s \in S^1_k[a,b]$ an \textit{interior knot}, when it lies between two neighboring data abscissae, that is, there is a $q$ with $0 \leq q \leq \mu$ and $x_q < t_j < x_{q+1}$. If $t_j$ coincides with $x_q$, that is, there is a $q$ with $1 \leq q \leq \mu$ and $t_j := x_q$, then $t_j$ ist called a \textit{data knot}. 
\begin{theorem}\label{satzEigenschaftenBesteApprox} 
Let $\mu \geq k+1 \geq 2$ and $1 \leq p \leq \infty$. For given data $(x_0,f_0),\ldots$, $(x_{\mu+1},f_{\mu+1})$ with $x_0 < x_1 < \ldots < x_{\mu} < x_{\mu+1}$ and $f_0,\ldots,f_{\mu+1} \in \mathbb{R}$ exists a best approximation $s^{\ast} \in S^1_k\left[x_0,x_{\mu+1}\right]$ 
\begin{equation*}
\big\|f-s^\ast\big\|_{p,X} \;\;=\;\; \inf_{s\in S^1_k[x_0,x_{\mu+1}]}\big\|f-s\big\|_{p,X}\;,
\end{equation*}
with the following additional features where $t_1,\ldots,t_k$ with $t_0 := x_0 < t_1 < \ldots < t_k < t_{k+1} := x_{\mu+1}$ denote the knots of $s^\ast$:
\begin{itemize}
  \item[(a)] There are no knots in the boundary regions. More precisely: 
\begin{equation*}
x_0 < x_1 \leq t_1 < \ldots < t_k \leq x_{\mu} < x_{\mu+1} \quad.
\end{equation*}
	
	\item[(b)] Data abscissae neighboring to interior knots are not knots.
	
	\item[(c)] Between two (not necessarily neighboring) interior knots of $s^\ast$ lie at least two data abscissae which are not knots of $s^\ast$.
	
	\item[(d)] On or between neighboring knots lie at least two data abscissae:
\begin{equation*}
\forall\, j,\; 0 \leq j \leq r: \;\; \exists\, i,\; 0\leq i \leq \mu: \;\; t_j \leq x_i < x_{i+1} \leq t_{j+1} \quad.
\end{equation*}

	\item[(e)] If an interior knot is situated between two neighboring data abscissae, then no additional knot lies on or between these data abscissae. That is, the proposition "`$t_j$ is a knot of $s^\ast$ with $x_q < t_j < x_{q+1}$"' implies $t_{j-1} < x_q$ and $t_{j+1} > x_{q+1}$.
	
	\item[(f)] Let $t_j$ be an interior knot. Then in each of the intervals $(-\infty,t_j)$ and $(t_j,\infty)$ there is a data abscissa from $\big\{x_1,\ldots,x_\mu\big\}$ which is not a knot.
	
	\item[(g)] For $p < \infty$ we have: Between an interior and a neighboring data knot of $s^\ast$ lies either exactly one data abscissa $x_q$ which is then reproduced ($t_j < x_q < t_{j+1}$, $s^\ast(x_q) = f_q$) or there exist at least two data abscissae $x_q, x_{q+1}$ between the knots ($t_j < x_q < x_{q+1} < t_{j+1}$).
	
	\item[(h)] All interior knots are proper knots, that is, the first derivative is discontinuous in all interior knots.
\end{itemize}
\end{theorem}
\begin{proof}
From Theorem \ref{existenzsatz} we know that at least one best approximation exists in $S^1_k[a,b]$. Let $s^\ast$ be a best approximation from $S^1_k[a,b]$ with minimal number of interior knots. The $r$, $0 \leq r \leq k$, proper knots of $s^\ast$ are denoted by $\tau_0 := x_0 < \tau_1 < \ldots < \tau_r < \tau_{r+1} := b$. In a first step we prove (a) to (g) under the additional assumption $r=k$:

\vspace{0.1in}
We prove (a) by contradiction. Assume $x_0 = \tau_0 < \tau_1 < x_1$. Then we can reduce the number of interior knots by one by replacing $s^\ast$ on $[x_0,x_1]$ by the straight line connecting $\big(x_0,s^\ast(x_0)\big)$ and $\big(x_1,s^\ast(x_1)\big)$. $\tau_1$ ceases to be interior point and $x_1$ becomes a data knot whithout change of the approximation quality. This is a contradiction to the assumed minimal number of interior knots. The case $x_\mu < \tau_r < x_{\mu+1}$ is lead to a contradiction analogously.

% ================
% Figure 8
% ================
\begin{figure}[h]
\begin{center}
%\fbox{
\begin{tikzpicture}
\draw (-2.5,0)--(3,0);
\draw (-2,0.1)--(-2,-0.1) node[below=2pt] {$x_0$};
\draw (0,0.1)--(0,-0.1) node[below=2pt] {$\tau_1$};
\draw (2,0.1)--(2,-0.1) node[below=2pt] {$x_1$};

\draw[fill] (-2,0.25) circle (2.0pt);
\draw[fill] (2,1.5) circle (2.0pt);

\draw[thick] (-2,0.25)--(0,1.5)--(3,0.5);
\draw[thick,dashed] (-2,0.25)--(2,0.8333);
\end{tikzpicture}
%} % END \fbox
\end{center}
\caption{Illustration of the proof of Theorem \ref{satzEigenschaftenBesteApprox} (a).}
\label{fig8}
\end{figure}
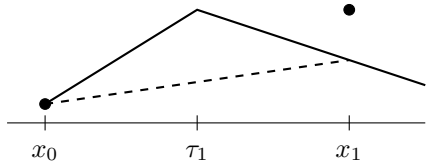

\vspace{0.2in}
To prove (b) let $\tau_j$ be an interior knot of $s^\ast$ and $x_q$ and $x_{q+1}$ knots neighboring to $\tau_j$, that is, $x_q < \tau_j < x_{q+1}$. If $x_q$ was a data knot we could replace $s^\ast$ on $[x_q,x_{q+1}]$ by the straight line connecting $\big(x_q,s^\ast(x_q)\big)$ and $\big(x_{q+1},s^\ast(x_{q+1})\big)$ whithout change of the approximation quality. At the same time this modification diminishes the number of interior knots by (at least) one which is a contradiction to the assumed minimal number of interior knots. The arguments for $x_{q+1}$ follow the same line.

% ====================
% Figure 9
% ====================
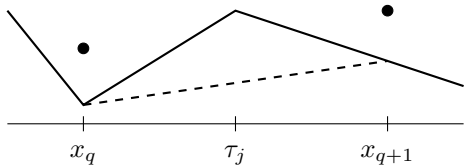
\begin{figure}[h]
\begin{center}
%\fbox{
\begin{tikzpicture}
\draw (-3,0)--(3,0);
\draw (-2,0.1)--(-2,-0.1) node[below=2pt] {$x_q$};
\draw (0,0.1)--(0,-0.1) node[below=2pt] {$\tau_j$};
\draw (2,0.1)--(2,-0.1) node[below=2pt] {$x_{q+1}$};

\draw[fill] (-2,1.0) circle (2.0pt);
\draw[fill] (2,1.5) circle (2.0pt);

\draw[thick] (-3,1.5)--(-2,0.25)--(0,1.5)--(3,0.5);
\draw[thick,dashed] (-2,0.25)--(2,0.8333);
\end{tikzpicture}
%} % END \fbox
\end{center}
\caption{Illustration of the proof of Theorem \ref{satzEigenschaftenBesteApprox} (b).}
\label{fig9}
\end{figure}

\vspace{0.2in}
To prove assertion (c) by contradiction we assume that assertion (c) does not hold. We distinguish the following cases:

\vspace{0.03in}
\underline{Case 1:} There lies no data abscissa between two interior knots $\tau_j$ and $\tau_{j'}$, that is, we have the situation $x_{q-1} < \tau_j < \tau_{j'} < x_q$. By replacing $s^{\ast}$ on $\left[x_{q-1},x_q\right]$ by the straight line connecting $\left(x_{q-1},s^{\ast}\left(x_{q-1}\right)\right)$ and $\left(x_q,s^{\ast}\left(x_q\right)\right)$ we get a new best approximation with (at least) two interior knots less. This contradicts the assumed minimal number of interior knots.

\vspace{0.03in}
\underline{Case 2:} Exactly one data abscissa $x_q$ is lying between two interior knots $\tau_j < \tau_{j'}$. Then $x_q$ cannot be a knot. Because otherwise we could replace $s^{\ast}$ between $x_{q-1}$ and $x_q$ by the straight line connection the function values $s^{\ast}\left(x_{q-1}\right)$ in $x_{q-1}$ and $s^{\ast}\left(x_q\right)$ in $x_q$ and could thus construct a best approximation with less interior knots than $s^\ast$ has. This contradicts the assumption that $s^\ast$ has a minimal number of interior knots. In analogy it can be shown that there ist neither a knot distinct from $x_q$ between $\tau_j$ and $\tau_{j'}$. Thus, we have $j' = j+1$ and $x_{q-1} < \tau_j < x_q < \tau_{j+1} < x_{q+1}$. 

\vspace{0.03in}
We definine
\begin{equation*}
\sigma(x) \;:=\; s^\ast\left(x_{q-1}\right) + \left(x - x_{q-1}\right) \frac{s^{\ast}\left(x_q\right)-s^{\ast}\left(x_{q-1}\right)}{x_q - x_{q-1}} \quad.
\end{equation*}

\vspace{0.03in}
\underline{Case 2.1:} We consider the case $s^{\ast}\left(\tau_j\right) < \sigma\left(\tau_j\right)$.

\vspace{0.03in}
\underline{Case 2.1.1:} Let the inequality $s^{\ast}\left(x_{q+1}\right) \leq \sigma\left(x_{q+1}\right)$ hold. Then there exists a point of intersection $x \in \left(\tau_{j+1},x_{q+1}\right]$ of $s^{\ast}$ with the straight line $\sigma$. We can replace $s^{\ast}$ on $\left[x_{q-1},x\right]$ by $\sigma$ and thus receive again a best approximation $\tilde{s}$ of $f$ with at least one interior knot less than $s^{\ast}$ because $\tau_j$ and $\tau_{j+1}$ are not interior knots for $\tilde{s}$, $x_{q-1}$ is data knot of $\tilde{s}$ and only $x$ is a new (potentially interior) knot of $\tilde{s}$. This contradicts our assumption that $s^\ast$ has the minimal number of interior knots.

% =====================
% Figure 10
% =====================
\begin{figure}[h]
\begin{center}
%\fbox{
\begin{tikzpicture}
\draw (-4,0)--(3,0);
\draw (-3,0.1)--(-3,-0.1) node[below=2pt] {$x_{q-1}$};
\draw (-1.5,0.1)--(-1.5,-0.1) node[below=2pt] {$\tau_j$};
\draw (-0.5,0.1)--(-0.5,-0.1) node[below=2pt] {$x_q$};
\draw (0.5,0.1)--(0.5,-0.1) node[below=2pt] {$\tau_{j+1}$};
\draw (1.75,0.1)--(1.75,-0.1) node[below=2pt] {$x$};
\draw (2.5,0.1)--(2.5,-0.1);
\draw (2.7,-0.1) node[below=2pt] {$x_{q+1}$};

\draw[->,>=stealth] (-1.75,-0.35)--(-2.5,-0.35);
\draw[->,>=stealth] (0.9,-0.35)--(1.55,-0.35);

\draw[fill] (-0.5,1) circle (1.5pt);
\draw[fill] (-3,1) circle (1.5pt);
\draw[fill] (2.5,0.55) circle (1.5pt);

\draw[thick] (-4,1.5)--(-1.5,0.25)--(0.5,1.75)--(3,0.25);
\draw[thick,dashed] (1.75,1)--(-3,1);
\end{tikzpicture}
%} % END \fbox
\end{center}
\caption{Illustration of the proof of Theorem \ref{satzEigenschaftenBesteApprox} (c), case 2.1.1.}
\label{fig10}
\end{figure}
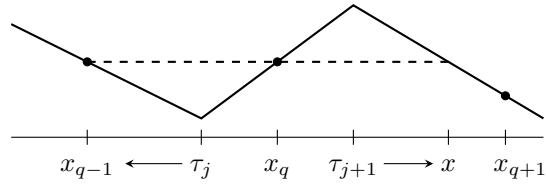

\vspace{0.1in}
\underline{Case 2.1.2:} We consider the case $s^{\ast}\left(x_{q+1}\right) > \sigma\left(x_{q+1}\right)$. By $\varphi$ we denote the connecting line between $\big(x_q,s^\ast(x_q)\big)$ and $(\tau_{j+1},s^\ast(\tau_{j+1})\big)$.

\vspace{0.03in}
\underline{Case 2.1.2.1:} $\varphi(x_{q+1}) > s^\ast(x_{q+1})$. Then there is a $x \in \big[x_{q-1},\tau_j\big)$ with $s^\ast(x) = \psi(x)$ where $\psi$ is the connecting line of the points $\big(x_q,s^\ast(x_q)\big)$ and $\big(x_{q+1},s^\ast(x_{q+1})\big)$. Replace $s^\ast$ on $\big[x,x_{q+1}\big]$ by $\psi$. Then $\tau_j$ and $\tau_{j+1}$ cease to be interior knots while $x$ and $x_{q+1}$ become new knots where $x_{q+1}$ is a data knot. The number of interior knots falls by at least one during this modification while the approximation power remains unchanged - contradicting the assumption on the minimal number of interior knots for the optimal $s^\ast$.

% =====================
% Figure 11
% =====================
\begin{figure}[h]
\begin{center}
%\fbox{
\begin{tikzpicture}
\draw (-4,0)--(3,0);
\draw (-3.5,0.1)--(-3.5,-0.1) node[below=2pt] {$x_{q-1}$};
\draw (-2.5,0.1)--(-2.5,-0.1) node[below=2pt] {$x$};
\draw (-1.5,0.1)--(-1.5,-0.1) node[below=2pt] {$\tau_j$};
\draw (0,0.1)--(0,-0.1) node[below=2pt] {$x_q$};
\draw (1.5,0.1)--(1.5,-0.1) node[below=2pt] {$\tau_{j+1}$};
\draw (2.5,0.1)--(2.5,-0.1);
\draw (2.7,-0.1) node[below=2pt] {$x_{q+1}$};

\draw[->,>=stealth] (-1.75,-0.35)--(-2.25,-0.35);
\draw[->,>=stealth] (1.85,-0.35)--(2.25,-0.35);

\draw[fill] (0,1) circle (1.5pt);
\draw[fill] (-3.5,1.0) circle (1.5pt);
\draw[fill] (2.5,1.5) circle (1.5pt);

\draw[thick] (-4,1.5)--(-1.5,0.25)--(1.5,1.75)--(3,1.0);
\draw[thick,dashed] (-2.5,0.75)--(2.5,1.25);
\end{tikzpicture}
%} % END \fbox
\end{center}
\caption{Illustration of the proof of Theorem \ref{satzEigenschaftenBesteApprox} (c), case 2.1.2.1.}
\label{fig11}
\end{figure}
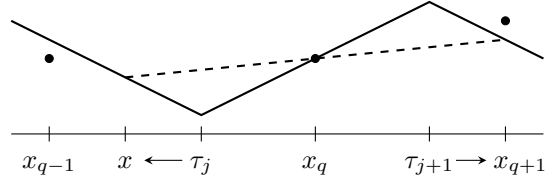

\vspace{0.2in}
\underline{Case 2.1.2.2:} The case $\varphi(x_{q+1}) = s^\ast(x_{q+1})$ cannot occur since $\tau_{j+1}$ would then not be a proper knot of $s^\ast$.

\vspace{0.03in}
\underline{Case 2.1.2.3:} Consider the case \; $\varphi(x_{q+1}) < s^\ast(x_{q+1})$. \; Let $\vartheta$ denote the straight line connecting $\big(x_{q-1},s^\ast(x_{q-1})\big)$ and $\big(\tau_j,s^\ast(\tau_j)\big)$. Then the lines $\vartheta$ and $\psi$ have a point of intersection $x \in (\tau_j,x_q)$. We replace $s^\ast$ on $[\tau_j,x]$ by $\vartheta$ and on $(x,x_{q+1})$ by $\psi$. By this modification the knot $\tau_j$ is moved to $x$ and the knot $\tau_{j+1}$ to $x_{q+1}$. This reduces the number of interior knots by one while the approximation power remains unchanged - contradicting the assumption on the minimal number of interior knots.

% ====================
% Figure 12
% ====================
\begin{figure}[h]
\begin{center}
%\fbox{
\begin{tikzpicture}
\draw (-4,0)--(3,0);
\draw (-3.5,0.1)--(-3.5,-0.1) node[below=2pt] {$x_{q-1}$};
\draw (-2.5,0.1)--(-2.5,-0.1) node[below=2pt] {$\tau_j$};
\draw (-1.5,0.1)--(-1.5,-0.1) node[below=2pt] {$x$};
\draw (0,0.1)--(0,-0.1) node[below=2pt] {$x_q$};
\draw (1.5,0.1)--(1.5,-0.1) node[below=2pt] {$\tau_{j+1}$};
\draw (2.5,0.1)--(2.5,-0.1);
\draw (2.7,-0.1) node[below=2pt] {$x_{q+1}$};

\draw[->,>=stealth] (-2.25,-0.35)--(-1.75,-0.35);
\draw[->,>=stealth] (1.85,-0.35)--(2.25,-0.35);

\draw[fill] (0,1) circle (2.0pt);
\draw[fill] (-3.5,1.0) circle (2.0pt);
\draw[fill] (2.5,1.5) circle (2.0pt);

\draw[thick,dashed] (-2.5,0.75)--(-1.5,0.25)--(2.5,2.25);
\draw[thick] (-4,1.5)--(-2.5,0.75)--(1.5,1.15)--(3,2.8);
\end{tikzpicture}
%} % END \fbox
\end{center}
\caption{Illustration of the proof of Theorem \ref{satzEigenschaftenBesteApprox} (c),  case 2.1.2.3.}
\label{fig12}
\end{figure}
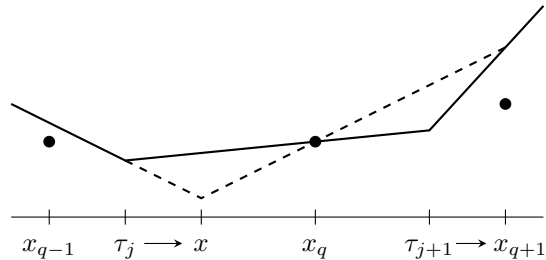

\vspace{0.2in}
\underline{Case 2.2:} The case $s^{\ast}\left(\tau_j\right) > \sigma\left(\tau_j\right)$ ist treated in analogy to case 2.1.

\vspace{0.03in}
\underline{Case 2.3:} The case $s^{\ast}\left(\tau_j\right) = \sigma\left(\tau_j\right)$ cannot occur because $\tau_j$ then would not be a proper knot and could therefore be omitted.

\vspace{0.03in}
\underline{Case 3:} We consider the case that at least two data abscissae lie between two interior knots $\tau_i$ and $\tau_j$, $i < j$, i.e.: $x_{q-1} < \tau_i < x_q < x_{q+1} < \ldots < x_{q'} < \tau_j < x_{q'+1}$ with $q < q'$. Then the data abscissae $x_q$ and $x_{q'}$ neighboring to $\tau_i$ and $\tau_j$ cannot be knots. Because otherwise we could replace $s^{\ast}$ on $\left[x_{q-1},x_q\right]$ by the straight line connecting $\left(x_{q-1},s^{\ast}\left(x_{q-1}\right)\right)$ and $\left(x_q,s^{\ast}\left(x_q\right)\right)$ and in a similar way on $\left[x_{q'},x_{q'+1}\right]$. We would thus receive a best approximation with less interior knots which contradicts our assumption. Thus there are at least two data abscissae which are not knots: $x_q$ and $x_{q'}$. This proves (c).

\vspace{0.2in}
To prove (d) we demonstrate that all other cases can be excluded:

\vspace{0.03in}
\underline{Case 1:} There is no data abscissa on or between neighboring knots $\tau_j < \tau_{j+1}$, i.e.: $x_q < \tau_j < \tau_{j+1} < x_{q+1}$. Then $\tau_j$ and $\tau_{j+1}$ must be interior knots and we can construct a best approximation with less interior knots then by replacing $s^\ast$ on $[x_q,x_{q+1}]$ by the straight line connecting $\big(x_q,s^\ast(x_q)\big)$ and $\big(x_{q+1},s^\ast(x_{q+1})\big)$. This contradicts the assumption that $s^\ast$ has the minimal number of interior knots.

% ====================
% Figure 13
% ====================
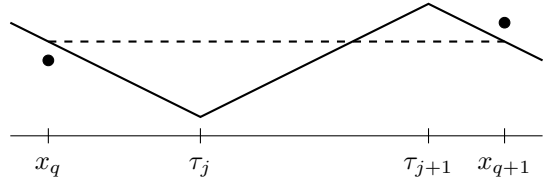
\begin{figure}[h]
\begin{center}
%\fbox{
\begin{tikzpicture}
\draw (-4,0)--(3,0);
\draw (-3.5,0.1)--(-3.5,-0.1) node[below=2pt] {$x_q$};
\draw (-1.5,0.1)--(-1.5,-0.1) node[below=2pt] {$\tau_j$};
\draw (1.5,0.1)--(1.5,-0.1) node[below=2pt] {$\tau_{j+1}$};
\draw (2.5,0.1)--(2.5,-0.1) node[below=2pt] {$x_{q+1}$};

\draw[fill] (-3.5,1.0) circle (2.0pt);
\draw[fill] (2.5,1.5) circle (2.0pt);

\draw[thick] (-4,1.5)--(-1.5,0.25)--(1.5,1.75)--(3,1.0);
\draw[thick,dashed] (-3.5,1.25)--(2.5,1.25);
\end{tikzpicture}
%} % END \fbox
\end{center}
\caption{Illustration of the proof of Theorem \ref{satzEigenschaftenBesteApprox} (d), case 1.}
\label{fig13}
\end{figure}

\vspace{0.2in}
\underline{Case 2:} There is exactly one data abscissa between $\tau_j$ and $\tau_{j+1}$.

\vspace{0.03in}
\underline{Case 2.1:} Consider the case $x_q = \tau_j < \tau_{j+1} < x_{q+1}$. We move $\tau_{j+1}$ to $x_{q+1}$ and replace $s^\ast$ on $[x_q,x_{q+1}]$ by the connecting line between $\big(x_q,s^\ast(x_q)\big)$ and $\big(x_{q+1},s^\ast(x_{q+1})\big)$. This diminishes the number of interior knots. Contradiction!

\vspace{0.03in}
\underline{Case 2.2:} The case $x_q < \tau_j < \tau_{j+1} = x_{q+1}$ is treated in analogy to case 2.1.

\vspace{0.03in}
\underline{Case 2.3:} The case $x_{q-1} < \tau_j < x_q < \tau_{j+1} < x_{q+1}$ cannot occur because of (c).

\vspace{0.03in}
Thus, the assertion in (d) remains as the only possibility.

\vspace{0.2in}
Proof of (e): We obviously have: $\tau_0 \neq \tau_j \neq \tau_k$. As data abscissae neighboring to $\tau_j$ the abscissae $x_q$ and $x_{q+1}$ cannot be knots according to (b).

\vspace{0.2in}
Proof of (f): Let $\tau_j$ be an interior knot of $s^\ast$. Because of (a) $\tau_j\in(x_1,x_\mu)$ holds and thus there exists a $q$ with
\begin{equation*}
x_0 < x_1 < \ldots < x_q < \tau_j < x_{q+1} < \ldots < x_\mu < x_{\mu+1} \quad.
\end{equation*}
Because of (b) $x_q$ and $x_{q+1}$ cannot be knots which implies
\begin{equation*}
x_0,x_q\in(-\infty,\tau_j) \quad\quad\text{and}\quad\quad x_{q+1},x_{\mu+1}\in(\tau_j,\infty) \quad.
\end{equation*}

\vspace{0.2in}
To prove (g) let $p < \infty$ hold. Because of (d) at least one data abscissa lies strictly between $\tau_j$ and $\tau_{j+1}$. Let there be exactly one $x_q$ with $\tau_j < x_q < \tau_{j+1}$.

\vspace{0.03in}
\underline{Case 1:} Consider the case that $\tau_j$ is the interior and $\tau_{j+1}$ the data knot. Then there is at least one data abscissa to the left of $\tau_j$. Denote by $\vartheta$ the line connecting the points $\big(\tau_j,s^\ast(\tau_j)\big)$ and $(x_{q-1},s^\ast(x_{q-1})\big)$. Assume the data point $(x_q,f_q)$ would not be reproduced exactly by $s^\ast$, i.e. $s^\ast(x_q) \neq f_q$. For small $\varepsilon > 0$ let $\psi_{+}$ be the straight line connecting $\big(\tau_j+\varepsilon,\vartheta(\tau_j+\varepsilon)\big)$ and $\big(\tau_{j+1},s^\ast(\tau_{j+1})\big)$ and $\psi_{-}$ the line connecting $\big(\tau_j-\varepsilon,\vartheta(\tau_j-\varepsilon)\big)$ and $\big(\tau_{j+1},s^\ast(\tau_{j+1})\big)$. Furthermore, we define
\begin{equation*}
s^\ast_{+}(x) \;\;:=\;\; \left\{
\begin{array}{lcl}
\vartheta(x) &,& x \in \big(x_{q-1},\tau_j+\varepsilon\big] \\
\psi_{+}(x) &,& x \in \big(\tau_j+\varepsilon,\tau_{j+1}\big] \\
s^\ast(x) &,& \text{otherwise} \hspace{2.0cm}\;
\end{array}
\right.
\end{equation*}
and 
\begin{equation*}
s^\ast_{-}(x) \;\;:=\;\; \left\{
\begin{array}{lcl}
\vartheta(x) &,& x \in \big(x_{q-1},\tau_j-\varepsilon\big] \\
\psi_{-}(x) &,& x \in \big(\tau_j-\varepsilon,\tau_{j+1}\big] \\
s^\ast(x) &,& \text{otherwise} \hspace{2.0cm}.
\end{array}
\right. %\;.
\end{equation*}
Then for sufficiently small $\varepsilon > 0$ and dependent on the position of $s^\ast(x_q)$ (see Fig. \ref{fig14} and \ref{fig15}) $s^\ast_{+}$ or $s^\ast_{-}$ is a better approximation to the data than $s^\ast$ - a contradiction since $s^\ast$ is a best approximation. 

\vspace{0.02in}
So if there ist exactly one $x_q$ lying between $\tau_j$ and $\tau_{j+1}$, i.e. $\tau_j < x_q < \tau_{j+1}$, then $x_q$ is reproduced by $s^\ast$, i.e. $s^\ast(x_q) = f_q$.

% =======================
% Figure 14
% =======================
\begin{figure}[h]
\begin{center}
%\fbox{
\begin{tikzpicture}
\draw (-4,0)--(3,0);
\draw (-3.5,0.1)--(-3.5,-0.1) node[below=2pt] {\small $x_{q-1}$};
\draw (-2.5,0.1)--(-2.5,-0.1) node[below=2pt] {\small $\tau_j$};
\draw (-1.5,0.1)--(-1.5,-0.1);
\draw (-1.3,-0.1) node[below=2pt] {\small $\tau_j+\varepsilon$};
\draw (0,0.1)--(0,-0.1) node[below=2pt] {\small $x_q$};
\draw (1.75,-0.1) node[below=2pt] {\small $\tau_{j+1}=$};
\draw (2.5,0.1)--(2.5,-0.1);
\draw (2.7,-0.1) node[below=2pt] {\small $x_{q+1}$};

\draw[->,>=stealth] (-2.25,-0.35)--(-1.75,-0.35);

\draw[fill] (0,1) circle (2.0pt);
\draw[fill] (-3.5,1.0) circle (2.0pt);
\draw[fill] (2.5,1.5) circle (2.0pt);

\draw[thick,dashed] (-2.5,0.75)--(-1.5,0.25)--(2.5,2.25);
\draw[thick] (-4,1.5)--(-2.5,0.75)--(2.5,2.25)--(3,2.8);
\draw (1,2.05) node {\small $s^\ast$};
\draw (1,1.2) node {\small $s^\ast_{+}$};
\end{tikzpicture}
%} % END \fbox
\end{center}
\caption{Illustration of the proof of Theorem \ref{satzEigenschaftenBesteApprox} (g),  case 1.}
\label{fig14}
\end{figure}

% =======================
% Figure 15
% =======================
\begin{figure}[h]
\begin{center}
%\fbox{
\begin{tikzpicture}
\draw (-4,0)--(3,0);
\draw (-3.5,0.1)--(-3.5,-0.1) node[below=2pt] {$x_{q-1}$};
\draw (-2.5,0.1)--(-2.5,-0.1) node[below=2pt] {$\tau_j-\varepsilon$};
\draw (-1.5,0.1)--(-1.5,-0.1);
\draw (-1.4,-0.1) node[below=2pt] {$\tau_j$};
\draw (0,0.1)--(0,-0.1) node[below=2pt] {$x_q$};
\draw (1.75,-0.1) node[below=2pt] {$\tau_{j+1} = $};
\draw (2.5,0.1)--(2.5,-0.1);
\draw (2.7,-0.1) node[below=2pt] {$x_{q+1}$};

\draw[->,>=stealth] (-1.6,-0.35)--(-2.0,-0.35);

\draw[fill] (0,1.5) circle (2.0pt);
\draw[fill] (-3.5,1.0) circle (2.0pt);
\draw[fill] (2.5,1.5) circle (2.0pt);

\draw[thick,dashed] (-2.5,0.75)--(2.5,2.25);
\draw[thick] (-4,1.5)--(-2.5,0.75)--(-1.5,0.25)--(2.5,2.25)--(3,2.8);
\draw (1,2.05) node {\small $s^\ast_{-}$};
\draw (1,1.2) node {\small $s^\ast$};
\end{tikzpicture}
%} % END \fbox
\end{center}
\caption{Illustration of the proof of Theorem \ref{satzEigenschaftenBesteApprox} (g),  case 1.}
\label{fig15}
\end{figure}

\vspace{0.2in}
\underline{Case 2:} In this case $\tau_j$ is the data knot and $\tau_{j+1}$ the interior knot. This case is treated in analogy to case 1.

\vspace{0.05in}
Thus, we have shown that one of the cases named in (g) occurs.

\vspace{0.2in}
\noindent{}Proof of (h): Since we are considering the case $r=k$ all knots, in particular all interior knots, are proper knots, i.e., the first derivatve has a jump discontinuity.

\vspace{0.2in}
\noindent{}In case $r=k$ ($s^\ast$ has $k$ proper knots) (a) to (h) are thus proven. 

\vspace{0.05in}
\noindent{}To complete the proof of Theorem \ref{satzEigenschaftenBesteApprox} we have to check the case $r < k$. First we show that $s^\ast$ cannot have (proper) interior knots then. To demonstrate this let us assume $\tau_j$ was an interior knot of $s^\ast$ and $x_q < \tau_j < x_{q+1}$. Replacing $s^\ast$ on $[x_q,x_{q+1}]$ by the straight line connecting $\big(x_q,s^\ast(x_q)\big)$ and $\big(x_{q+1},s^\ast(x_{q+1})\big)$ makes $x_q$ and $x_{q+1}$ to (new) knots while $\tau_j$ ceases to be a knot. We thus get a best approximation whith $r+1\leq k$ knots and with one \textit{inner} knot less than $s^\ast$. This contradicts the assumption of $s^\ast$ having a minimal number of interior knots. All knots $\tau_1,\ldots,\tau_r$ of $s^\ast$ are therefore data knots.

\vspace{0.1in}
Furthermore,
\begin{equation*}
M \;:=\; \big\{x_1,\ldots,x_{\mu}\big\}\setminus\big\{\tau_1,\ldots,\tau_r\big\}
\end{equation*}
contains at least $k-r+1$ elements since $\mu\geq k+1$. Now supplement the proper knots of $s^\ast$ with arbitrary points from $M$ to $k$ (proper and improper) knots. We claim that for these now $k$ knots $t_1 < t_2 < \ldots < t_k$ properties and assertions (a) to (g) hold. Assertion (a) holds because of the way the additional knots were chosen. Implications (b), (c) (e), (f), (g) and (h) hold because the assumptions cannot be fulfilled due to the lack of interior knots. Assertion (d) holds, because alle knots $t_1,\ldots,t_k$ are data knots - the proper and the improper ones.

\vspace{0.1in}
This completes the proof of Theorem \ref{satzEigenschaftenBesteApprox}. \hfill $\square$
\end{proof}

\vspace{0.3in}
With regard to the application of the propositions in this paper to the development of numerical methods, we hold that:
\begin{remark}\label{bemerkungInterpolation}
The statements in Lemma \ref{minimalfolgenlemma}, Theorem \ref{existenzsatz} and Theorem \ref{satzEigenschaftenBesteApprox} (a) to (g) also hold true for the case
\begin{equation*}
\inf_{s\in S^1_k[a,b]}\big\|F-s(X)\big\|_p \;\;=\;\; 0 \quad.
\end{equation*}

\vspace{0.1in}
For Lemma \ref{minimalfolgenlemma} and Theorem \ref{existenzsatz} this is obvious, for Theorem \ref{satzEigenschaftenBesteApprox} a review of the proof shows that nowhere it is assumed or required that the minimal approximation error be (strictly) positive.
\end{remark}

\vspace{0.2in}
\begin{remark}
As a review of the proof of Theorem \ref{satzEigenschaftenBesteApprox} shows the assumption \mbox{$\mu \geq k+1 \geq 2$} is needed only in case the minimal approximation error can already be realized with less than $k$ knots.
\end{remark}

%%%%%%%%%%%%%%%%%%%%%%%%%%%%%%
%%%%%%%%%%%%%%%%%%%%%%%%%%%%%%
% Summary
%%%%%%%%%%%%%%%%%%%%%%%%%%%%%%
%%%%%%%%%%%%%%%%%%%%%%%%%%%%%%
\vspace{0.2in}
\setcounter{equation}{0}
\section{Summary and Conclusion}
\label{sec:summary}
We have proven that there exists a first-degree spline (broken line) best approximating a function $f$ on a discrete point set and derived additional important properties of at least one best approximation. A numerical procedure based on the existence proof and the additional properties is presented in \cite{cromme01}. 

\vspace{0.075in}
On the question of extending the results from Theorem \ref{satzEigenschaftenBesteApprox} to higher-degree splines: Ist seems not possible to transfer the method of proof applied here (of varying a best approximation in subintervalls) to higher-degree splines, because in that case it would be necessary to preserve not just the continuity of the spline but also the continuity of the first and possibly higher derivatives during this modification.

%%%%%%%%%%%%%%%%%%%%%%%%%%%%%%%%%
%%%%%%%%%%%%%%%%%%%%%%%%%%%%%%%%%
% References
%%%%%%%%%%%%%%%%%%%%%%%%%%%%%%%%%
%%%%%%%%%%%%%%%%%%%%%%%%%%%%%%%%%

\end{document}